\documentclass[10pt]{article}

\usepackage[top=1in, bottom=1in, left=1in, right=1in]{geometry}
\usepackage[font=footnotesize,labelfont=bf]{caption}
\usepackage{graphicx}
\usepackage{epstopdf}
\usepackage{calc}
\usepackage{amsbsy,amsfonts,amsmath,amssymb,enumerate,epsfig,rotating,subfigure}
\usepackage[numbers,sort&compress]{natbib}
\usepackage[flushleft]{threeparttable}
\usepackage{amsthm}
\usepackage[pdfborder={0 0 0},colorlinks=true,allcolors=blue]{hyperref}
\usepackage{url}

\usepackage{multirow}
\usepackage[nofillcomment, vlined, ruled, algosection, algo2e]{algorithm2e}
\usepackage{color}
\usepackage{mathtools}

\newtheorem{theorem}{Theorem}[section]
\newtheorem{lemma}[theorem]{Lemma}
\newtheorem{corollary}[theorem]{Corollary}

\DeclarePairedDelimiter\step{\langle}{\rangle}

\graphicspath{{figures/}}

\begin{document}

\title{Developing a Mathematical Model for Bobbin Lace}

\author{
Veronika Irvine\thanks{Corresponding author. Email: vmi@uvic.ca} and Frank Ruskey\thanks{Research supported in part by an NSERC Discovery Grant}\\
\vspace{6pt}
\em{Department of Computer Science}\\
\em{University of Victoria, British Columbia, Canada}
}

\maketitle

\begin{abstract}
\noindent Bobbin lace is a fibre art form in which intricate and delicate patterns are created by braiding together many threads.  An overview of how bobbin lace is made is presented and illustrated with a simple, traditional bookmark design. Research on the topology of textiles and braid theory form a base for the current work and is briefly summarized. We define a new mathematical model that supports the enumeration and generation of bobbin lace patterns using an intelligent combinatorial search.  Results of this new approach are presented and, by comparison to existing bobbin lace patterns, it is demonstrated that this  model reveals new patterns that have never been seen before.  Finally, we apply our new patterns to an original bookmark design and propose future areas for exploration. \bigskip

00A06; 00A66; 05B45; 05C63; 05C20
\bigskip
\end{abstract}

Bobbin lace is a fibre art form constructed by braiding together many threads.  A very small set of actions is used in its production but the plethora of ways in which these actions can be combined results in the complex organization of threads into lace.  Over the past 500 years,  lacemakers have explored this rich domain relying primarily on trial and error.  The goal of our research is to develop a mathematical model as a systematic way of examining the myriad of possibilities.

In this paper we start with an overview of bobbin lace: what it is and how it is made.  A simple, traditional pattern is provided for the reader who wishes to experiment with this art form.  We then look at related work on the topology of textiles, its application to lace and key ideas that form a basis for our model. The main contribution of this paper is the definition of a model that uses braid theory and graph theory to describe workable patterns.  An intelligent combinatorial search algorithm is used to enumerate and exhaustively generate patterns consistent with this model, and the results are discussed.  Finally, we provide a second pattern for the reader to try.  This original design makes use of new, algorithmically generated patterns which we believe have not been previously discovered.

\section{Introduction}
\begin{figure*}
\begin{minipage}{\textwidth}
\begin{center}
\subfigure[]{
\resizebox*{!}{3.5cm}{\includegraphics{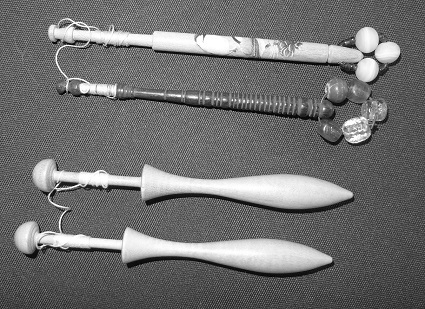}}}%
\subfigure[]{
\resizebox*{!}{3.5cm}{\includegraphics{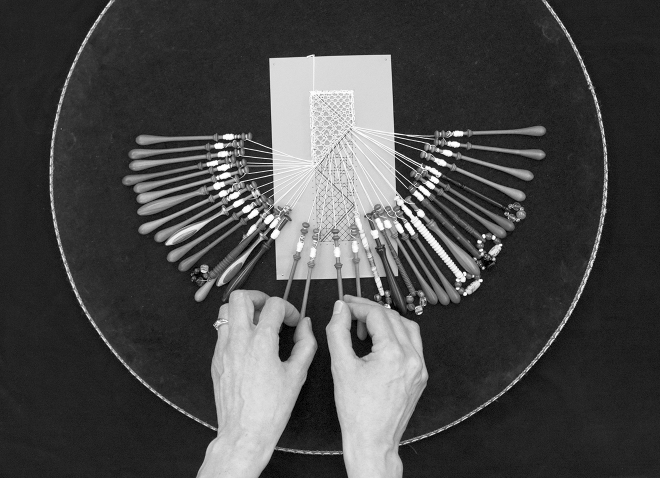}}}%
\subfigure[]{
\resizebox*{!}{3.5cm}{\includegraphics{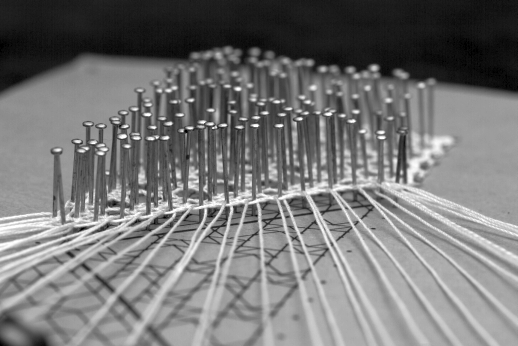}}}%
\caption{a) Two different bobbin styles.  b) Lace bookmark in progress on pillow. c) Closeup of pins.}
\label{fig:materials}
\end{center}
\end{minipage}
\end{figure*}

\label{sec:lacemaking}
There are many ways to create lace by hand.  Some are still commonly practiced in homes today, such as crocheting, knitting and tatting, but there are many rarer forms such as the drawn and cut thread style of Hardanger \cite{hardanger}, the layered embroidery approach of Carrickmacross \cite{carrickmacross} and the subject of this paper, bobbin lace.  Bobbin lace is made by braiding threads together; typically 50 to 100 threads are used but sometimes as many as several hundred.
It has a long history in Europe including the United Kingdom \cite{levey}.  Believed to have developed from passementerie, the art of ornamental braiding, it evolved into its current form during the last half of the 15th and first half of the 16th centuries.  Throughout its history, bobbin lace has been used to create delicate and complex designs with applications including furniture trimmings and coverings (table cloths, bed spreads, skirts, doilies), religious ornaments (altar cloths, robe overlays) and many articles of fashion (shawls, collars, inserts and edgings).  The fortunes of hand-made lace have been subject to the dictates of fashion, war and technology.  In the 18th century it was as valuable as gold and supported a thriving industry employing primarily impoverished women and girls. In the 19th century, stiff competition from machine-made copies drove down the price and led to a simplification of designs.  After World War I, fashions and the expectations of women in the work force changed.  Hand-made bobbin lace production as an industry ceased completely and was relegated to the status of hobby craft while machine-made lace was primarily used for curtains.  Interest in reviving the craft started in the 1970's and initially focused on the simple designs produced in the 19th century \cite{halley, laceguildhistory}.  Over the past 20 years, interest has started to turn toward the more advanced techniques employed in early laces as well as the invention of new techniques.  Active discussion and development continues at the many lace guilds \cite{ioli, oidfa, ukguild} and online discussion groups \cite{ning, snl}.

Perhaps the easiest way to understand bobbin lace is to look at the six step process by which it is made.

\textbf{Step 1) Prepare the threads.}  To manage many long threads without creating a tangled mess, each end of a thread is wound evenly onto one of a pair of bobbins (see Figure \ref{fig:materials}a).  A bobbin, commonly made from wood, is about 10cm in length.  One end is flanged to hold a length of thread and the other end, usually thicker and sometimes weighted with beads, is the handle.

\textbf{Step 2) Prepare the pattern.}  The lace is worked on top of a firm pillow stuffed with material such as straw, sawdust, wool, or, in some modern pillows, ethafoam or polystyrene. The pillow can be disk shaped (cookie pillow, see Figure \ref{fig:materials}b) or sausage shaped (bolster pillow).   As threads are braided together, they are held in place by brass pins pushed into the pillow (see Figure \ref{fig:materials}c).  To start a piece of lace, a pattern, such as the one shown in Figure \ref{fig:oldpattern}a, is copied onto stiff material - originally this may have been vellum or birch bark but modern lace makers use coloured card stock or printing paper overlaid with blue contact paper.  The black dots in the pattern represent the position of pins.  Before starting to make the lace, all of these dots are pricked through to make small holes.  The pattern is then pinned to the pillow.

\textbf{Step 3) Hang the bobbins.} The middle of each thread is draped around an anchoring pin at the top of the pattern with the pair of bobbins hanging down on either side. Often the first row of the pattern is a simple weave to anchor the threads.

\begin{figure*}
\begin{center}
\def\svgscale{0.3}
    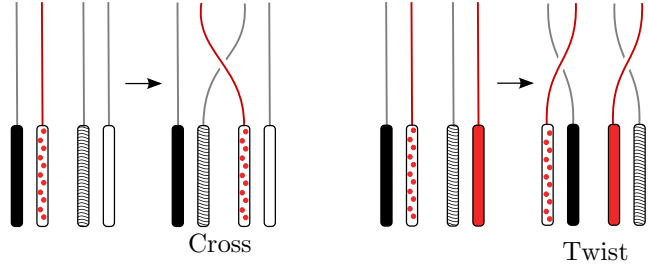 
\caption{Cross and twist: the two base actions used in bobbin lace.}
\label{fig:crosstwist}
\end{center}
\end{figure*}

\begin{figure*}
\begin{center}
\resizebox*{!}{5cm}{\includegraphics{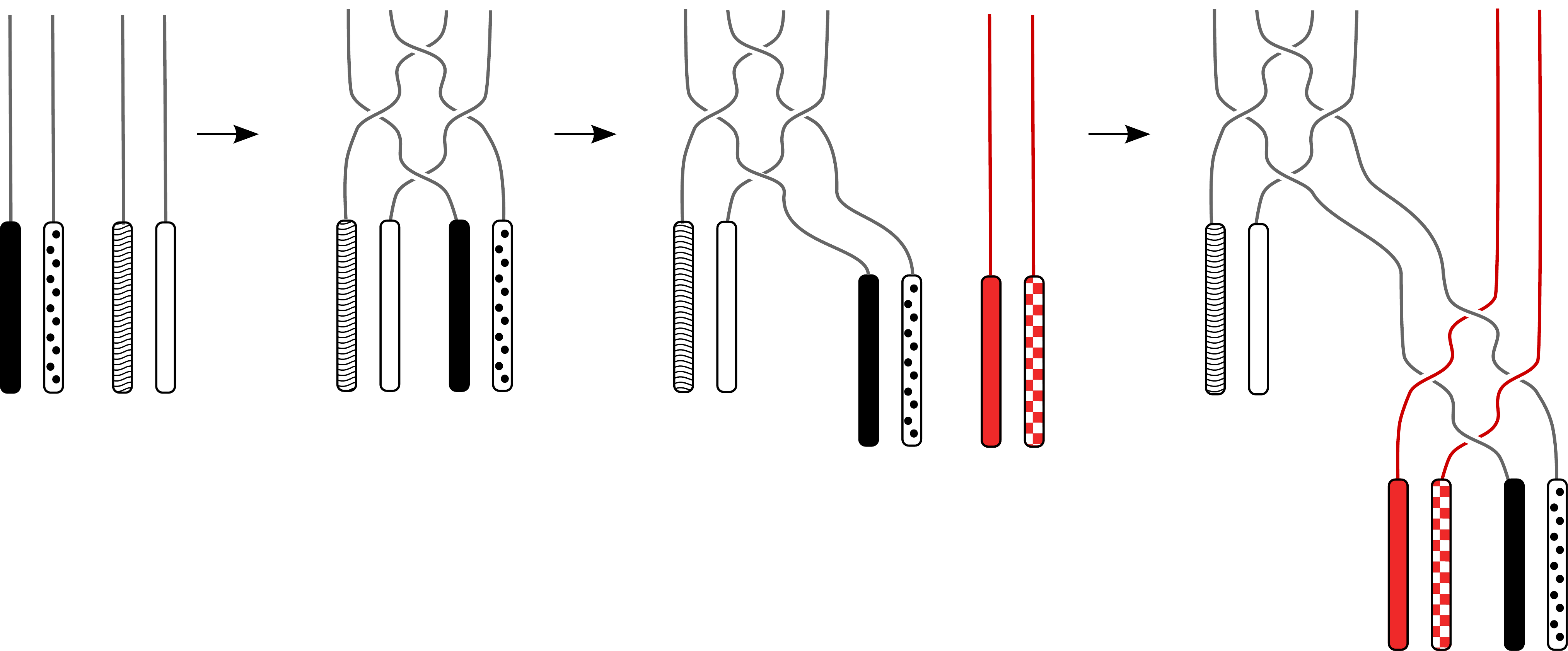}}
\caption{Progression of pairs of threads from one set of actions to another.}
\label{fig:traversal}
\end{center}
\end{figure*}

\textbf{Step 4) Braid and pin.} The lacemaker braids the threads working with four consecutive threads at a time. It is important to note that the four threads are treated as two pairs of threads: a left and a right pair.  Braids are made using two very simple actions.  The first action, known as a `cross' (which we will represent as $C$ where brevity is required), is performed by taking the rightmost thread from the left pair and crossing it over the leftmost thread of the right pair (see Figure \ref{fig:crosstwist}). The second action, known as a `twist' (denoted $T$), is performed by crossing the rightmost thread of the left pair over the leftmost thread of the left pair and similarly crossing the rightmost thread of the right pair over the leftmost thread of the right pair.  Occasionally, a variation of the twist is used: a `left-twist' (in which only the left pair is twisted) or a `right-twist' (in which only the right pair is twisted).
During a sequence of braiding actions, the lacemaker may insert a pin to hold the braid in place (see Figure \ref{fig:materials}c).  The pin provides resistance so that the lacemaker can apply tension to an individual thread without distorting its neighbours. The pinning action (denoted $p$) is performed by placing a pin between the left and right pairs of threads and into one of the prepared holes. Pinning may take place either in the middle of a braid sequence (after which the pin is `closed' because it is enclosed by threads) or after the braid (an `open' pin).
A lace braid can be made from any combination of these actions.  For example, $CTCTCT$ produces a four stranded plait similar to the three stranded plait commonly used to braid hair.  Similarly, $CTpCT$, $CpCT$, $TTT$, $TCp$, $C$ are all valid braids.  The exact sequence of actions used by the lacemaker depends on the pattern.

\textbf{Step 5) Advance to next set.} Once four threads have been braided and pinned, the lace maker moves on to braid another set of four consecutive threads.  This new set of threads may include two threads from the previous set, as shown in Figure \ref{fig:traversal}, but it may also be formed from four completely new threads.  Which four threads comprise the next set depends on the pattern.

Steps 4 and 5 are repeated until the lace pattern is completed.

\textbf{Step 6) Finish.} The threads are secured (sometimes with a knot, sometimes by weaving them back into the lace) and trimmed off.  The pins are removed and the lace may be lifted off the pillow.  Once the pins are removed, the lace is held together by the over and under crossings of the threads and friction.  Like knitted or woven material, if an individual thread is snagged and pulled away from the rest of the piece, the lace will distort.  However, as long as force is not too great and is applied over a number of threads, well made lace will maintain its shape and arrangement.

\begin{figure*}
\begin{center}
\begin{minipage}{\textwidth}
\begin{center}
\subfigure[]{
 \def\svgscale{0.4}
    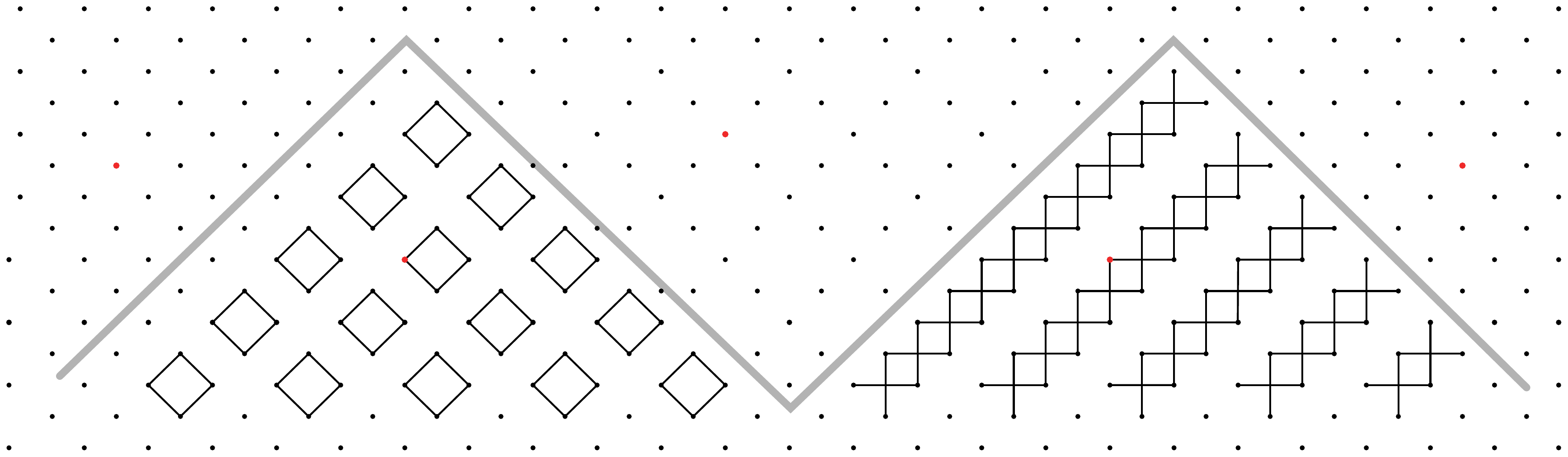 
}
\vspace{1cm}
\subfigure[]{
 \def\svgscale{0.4}
    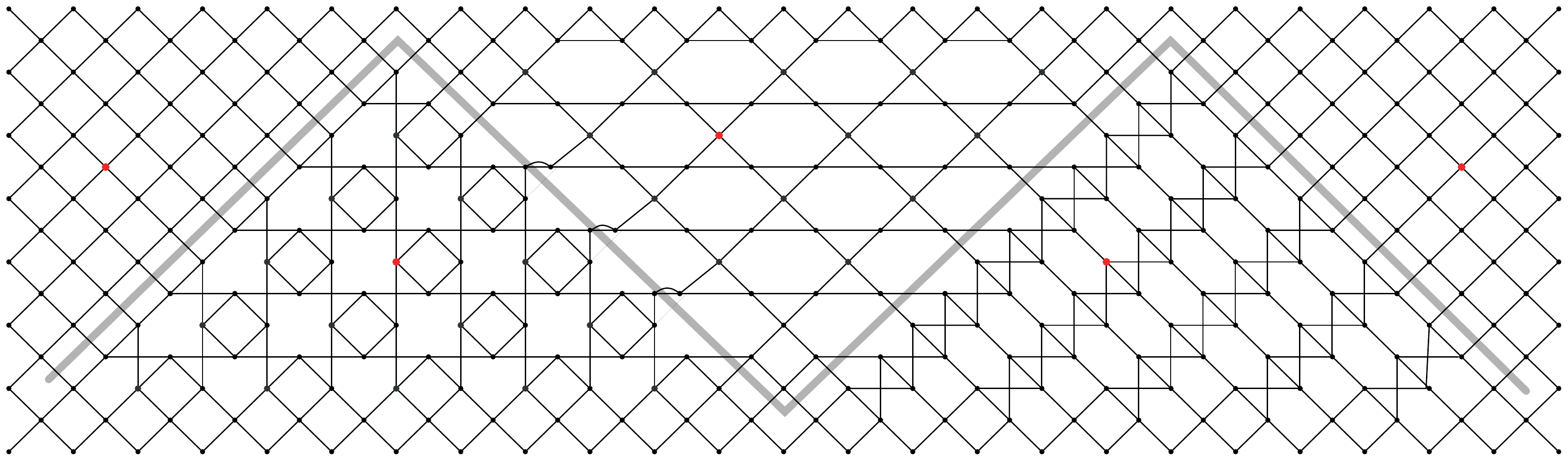 
}
\vspace{1cm}
\subfigure[]{
 \def\svgscale{0.7}
    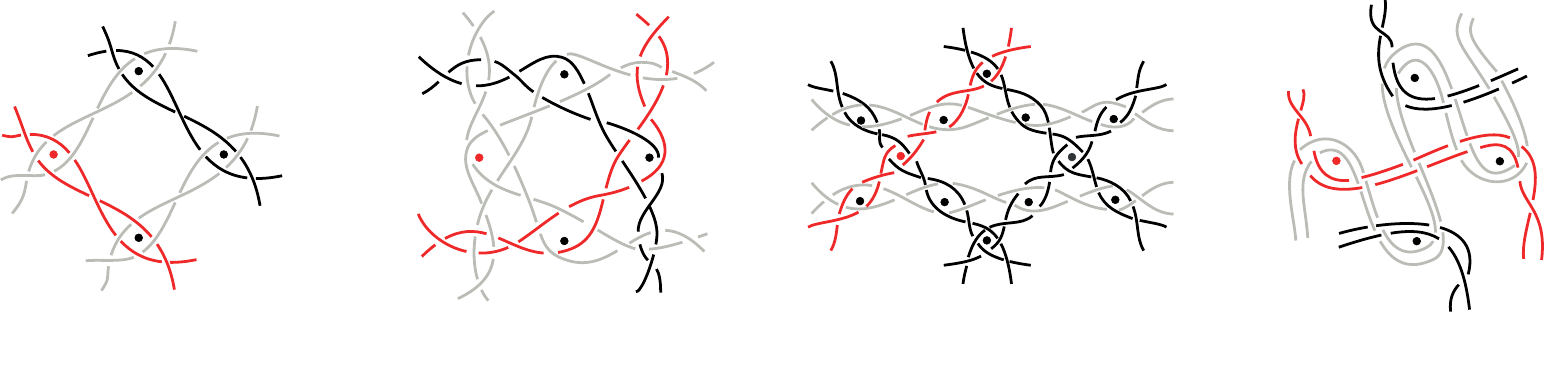 
}
\vspace{1cm}
\subfigure[]{
\resizebox*{\textwidth}{!}{\includegraphics{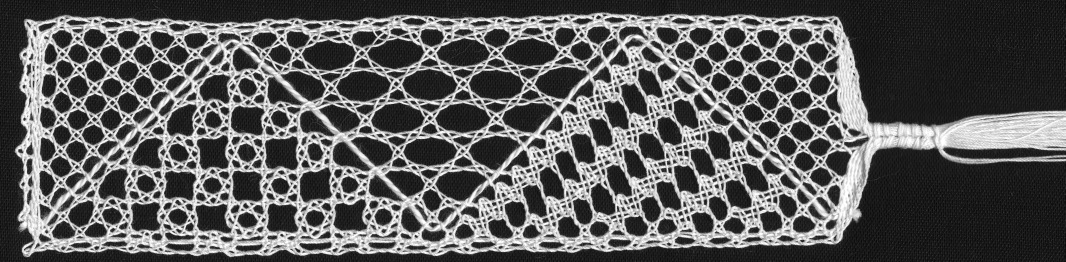}}} 
\end{center}
\caption{Note: Diagrams have been rotated $90^{\circ}$ counter clockwise. a) A bookmark pattern using traditional grounds.
b) Pair working diagram with regions labelled A, B, C, D and A. Grounds used in the corresponding regions are A) Torchon ground B) Rose ground C) Paris ground D) Bias ground.
c) Thread working diagrams.
d) Completed bookmark.  Actual size: 4 cm by 15 cm.  Made with 19 pairs of \emph{DMC Cordonnet Special No. 40.} primary thread and \emph{DMC Coton Perl\'{e} No 5.} gimp thread.}
\label{fig:oldpattern}
\end{minipage}
\end{center}
\end{figure*}

Many aspects of the process depend on the `pattern', so we shall take a closer look and discuss how it is interpreted by the lacemaker.  The pattern (e.g., see Figure \ref{fig:oldpattern}a) is always in the form of a diagram but the information it contains can vary quite a bit.  The position of a pin is indicated by a dot.  Sometimes a decorative thread of a contrasting colour or thickness, known as a \emph{gimp}, is used to outline a region.  The path of the gimp is marked with a bold line (e.g., Figure \ref{fig:oldpattern}a contains one gimp drawn as a thicker grey line).  In some patterns, lines are drawn in areas where the thread paths are complicated or ambiguous.  In these cases, a line segment represents two threads. Most patterns are accompanied by additional instructions in the form of text or a working diagram. The text often refers to a known `ground' pattern.

A \emph{ground} is a small pattern that can be repeated by periodic tiling to fill a closed region. Lace from a specific geographical area is often characterized by the use of a particular set of grounds well known to the lacemakers of that region.   A ground has specific instructions for the actions performed at each pair crossing, the placement of pins and the order in which pairs are combined.  Many grounds have been catalogued \cite{cook, viele, system} and commonly used grounds are described in most lace reference books.

A \emph{working diagram} is a line drawing that illustrates  how a sub-section of the lace is worked.  Working diagrams may depict individual threads (Figure \ref{fig:oldpattern}c) or pairs of threads (Figure \ref{fig:oldpattern}b). In working diagrams, cross and twist action combinations are often indicated by colouring the lines (using the International Colour Coding System \cite{colourcode}) or decorating the lines with hatch marks.

This has been a brief overview.  For more detailed instructions on equipment and technique, the reader is referred to websites such as \cite{edkins, halleylearn} or books such as \cite{nottingham, dye}.

\section{Related Work}
\label{sec:background}

The application of mathematical modeling to fibre arts is a fairly new area of research, with most of the focus on weaving and knitting. In the introduction to \cite{belcastroyackel}, Belcastro and Yackel give a nice overview of its history.  As far as we are aware, there have not been any publications specific to the mathematical exploration of bobbin lace.  However, more general work on the topology of threads in textiles \cite{grishanovPart1} is directly applicable.  Similarly, Artin's theory of braids \cite{artin}, while only remotely inspired by textiles and lace, gives a precise way of describing an alternating braid which is key to the structure of bobbin lace.  We must also acknowledge the work of modern lacemakers who have systematically explored ways to create new lace grounds.

\subsection{Topology of Textiles}
\label{sec:topology}

\begin{figure*}
\centering
\def\svgscale{0.5}
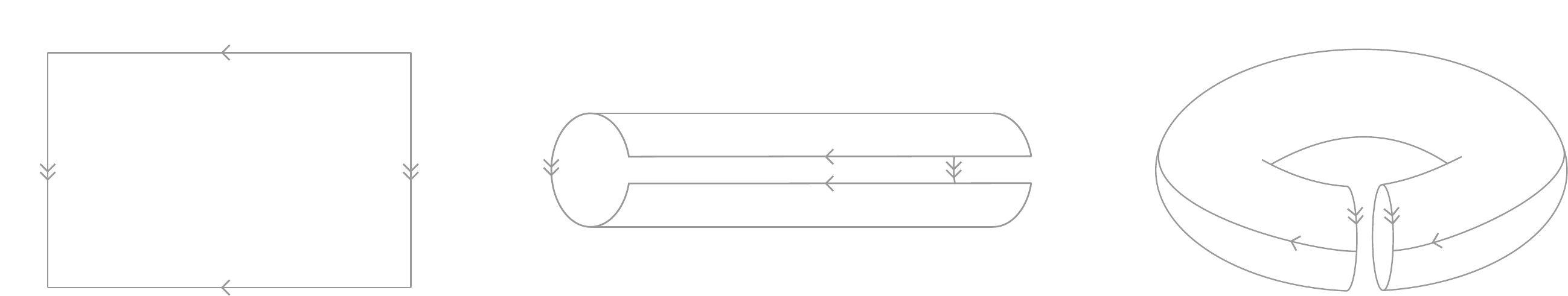
\caption{Left: Parallelogram with edge markings.  Middle and Right: Bending of parallelogram to form a torus.}
\label{fig:torus}
\end{figure*}

\begin{figure*}
\begin{center}
\begin{minipage}{\textwidth}
\begin{center}
\subfigure[]{
\resizebox*{!}{3cm}{\includegraphics{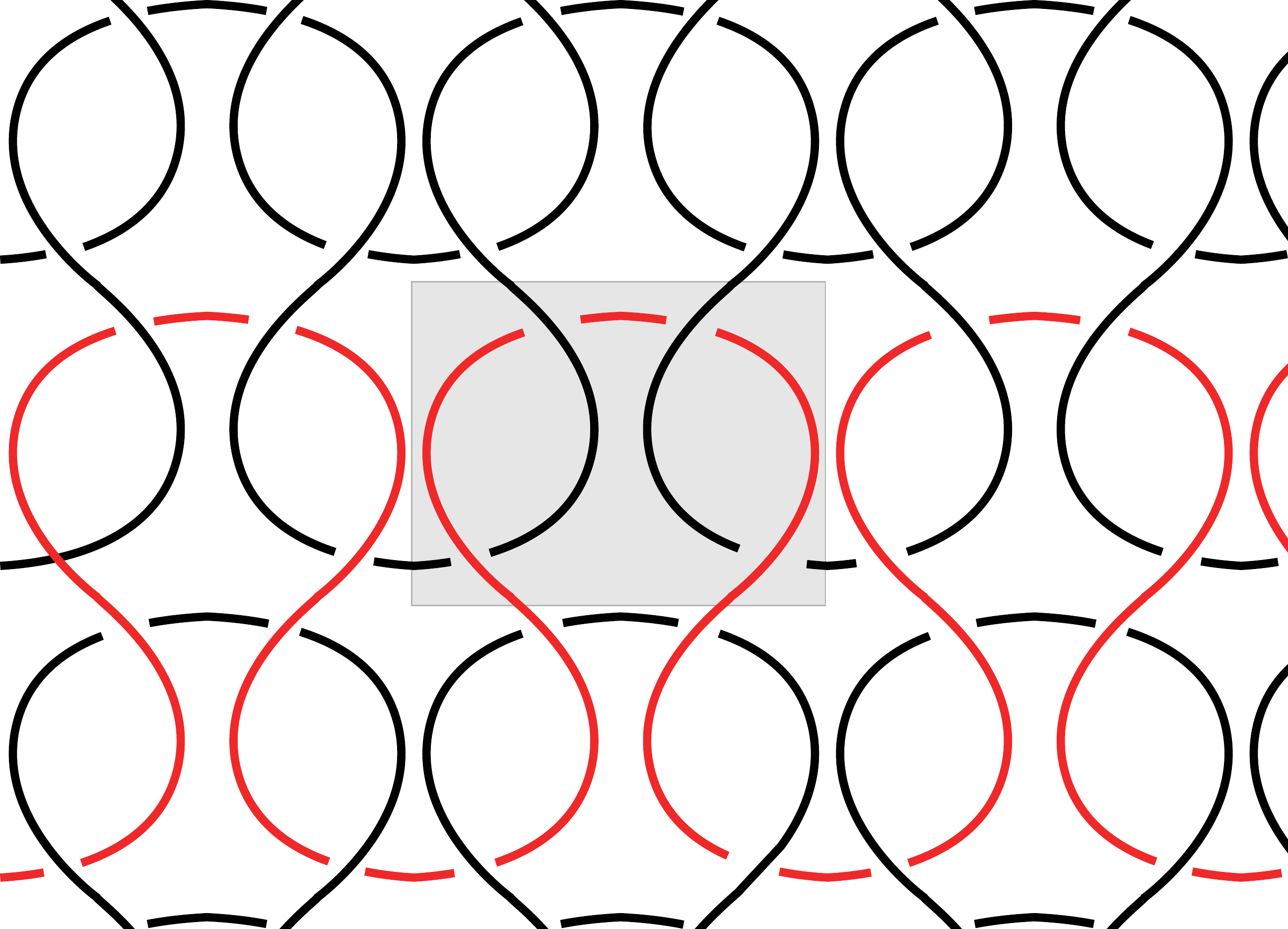}}} 
\hspace{0.5cm}
\subfigure[]{
\resizebox*{!}{3cm}{\includegraphics{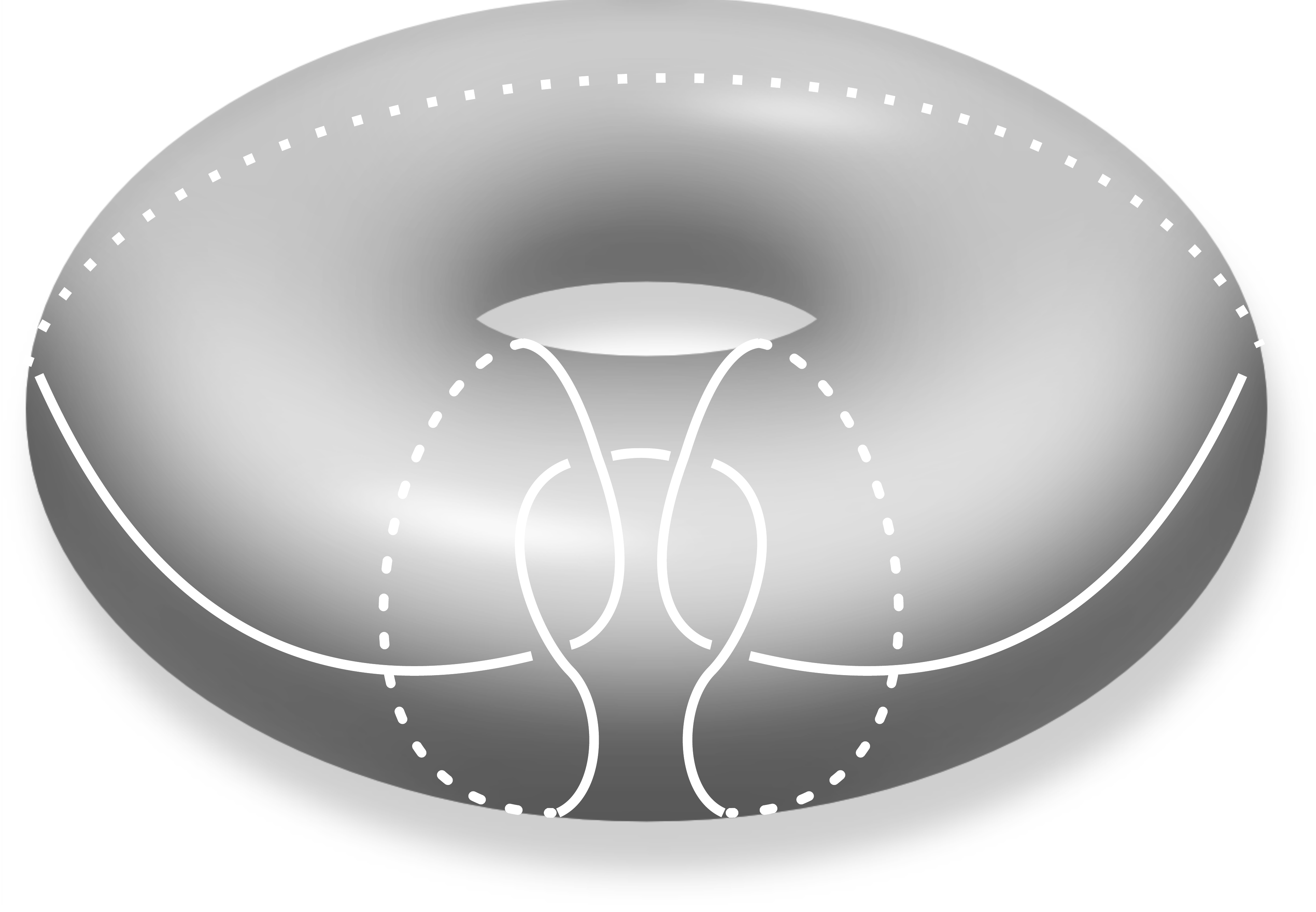}}}
\hspace{0.5cm}
\subfigure[]{
\resizebox*{!}{3cm}{\includegraphics{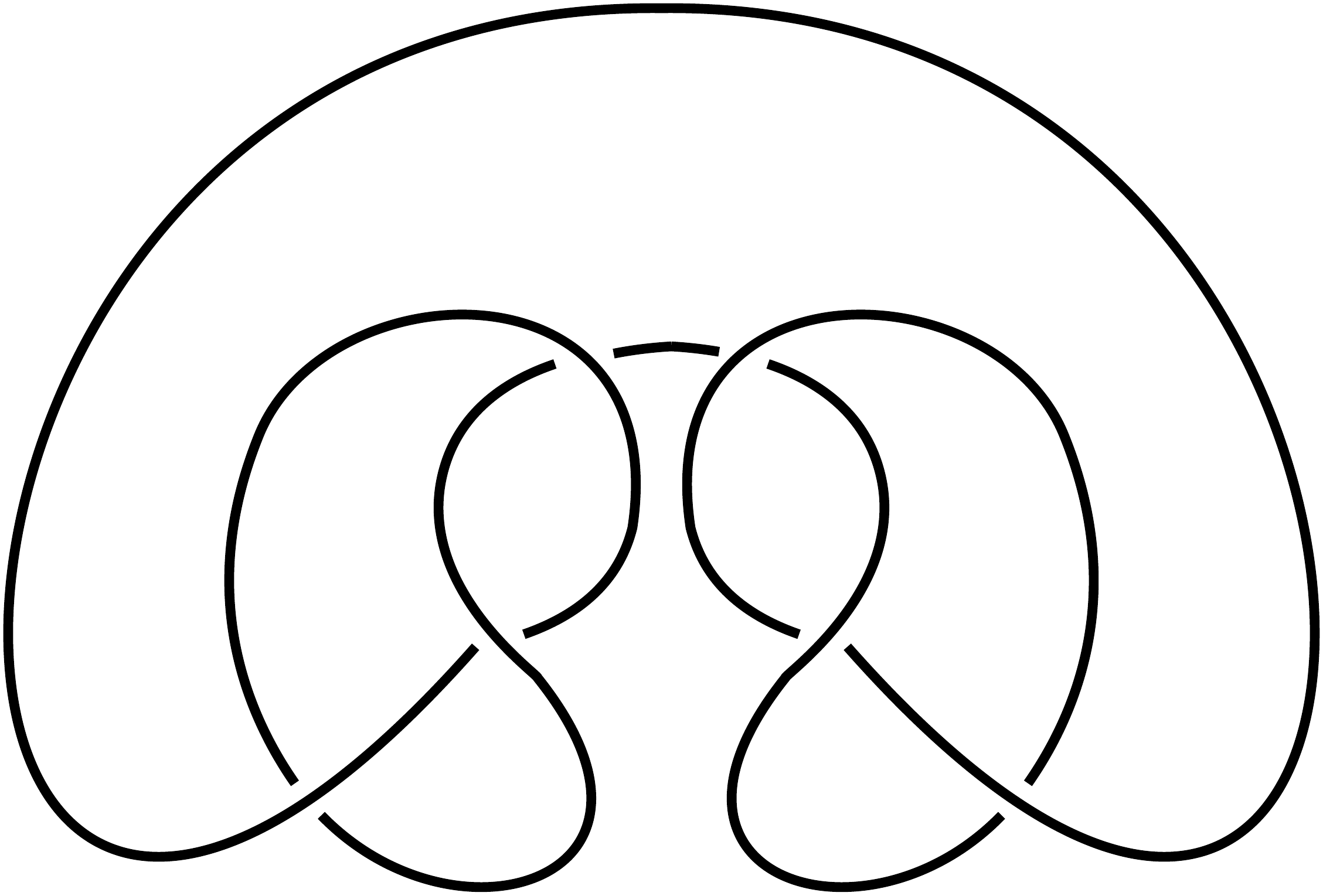}}}
\end{center}
\caption{`Knit' action represented as a) a period parallelogram, b) wrapped around a torus, and c) the associated knot.  Based on drawings by Grishanov \mbox{et al.} \cite{grishanovPart1}.}
\label{fig:grishanov}
\end{minipage}
\end{center}
\end{figure*}

Grishanov, Meshkov and Omelchenko  have examined the structure of machine-made textiles and classified these textiles using the ambient isotopy invariant of knots \cite{grishanovPart1}.  Textiles are typically made by repeating an arrangement of fibers in a periodic manner to cover an indefinitely large area.  The arrangement of fibers can be represented as a period parallelogram which is translated in two non-parallel directions to create an edge-to-edge tiling of the plane \cite{tilingsbook}.  Periodic repetition in textiles is a stronger property than just simple translation of a wallpaper decoration: fibers that terminate at the edges of the parallelogram must connect with fibers of adjacent copies.  This property, which Grishanov \mbox{et al.} call doubly periodic, can be visualized by joining opposite edges of the period parallelogram to form a torus (see Figure \ref{fig:torus}).    When wrapped around a torus, the fibers connect, forming a knot or a link as shown in Figure \ref{fig:grishanov}.  The toroidal representation also reduces a pattern description from infinite to finite size without loss of information, a key idea which we will revisit when describing our own model in Section \ref{sec:pairtraversal}.

\subsection{Braid Theory}
\label{sec:braids}

\begin{figure*}
\begin{center}
 \def\svgscale{0.4}
    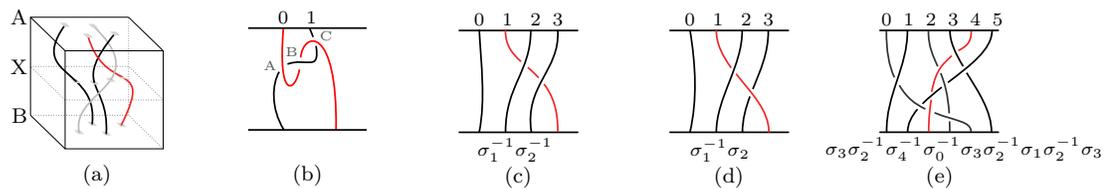
\caption{a) 3D braid between planes $A$ and $B$. Plane $X$ intersects each strand exactly once.  2D projections: b) Not a braid because the first strand can not be made monotonic without breaking the montonicity of the second strand, c) A non-alternating braid, d) and e) Alternating braids.}
\label{fig:braids}
\end{center}
\end{figure*}

The topology of bobbin lace has a direct relationship with braid theory.  With some minor exceptions\footnote{Techniques such as `sewings' do not produce braids because the threads cross back upon themselves.  Patterns involving a `lazy' crossing (a crossing in which two or more consecutive threads are treated as one and cross over or under other threads as a group) result in braids that are not alternating.  These techniques are outside of the scope of this paper.}, bobbin lace grounds are themselves braids, and, more specifically, they are alternating braids.  A \emph{braid} is defined mathematically as a set of $n$ `strands', each of which can be described as a curve in $\mathbb R ^3$.  The strands travel between two horizontal planes, $A$ and $B$, such that i) each strand originates at a unique point on plane $A$ and terminates at a unique point on plane $B$, ii) strands do not intersect one another or themselves, and iii) each strand is monotonic in the direction of a vertical line meaning that any horizontal plane $X$ between $A$ and $B$ will intersect each strand at just one point (see Figure \ref{fig:braids}a).  Braids are often represented as a 2D projection of the 3D object such that the start and end horizontal planes appears as horizontal lines. When two strands cross or wrap around each other, one strand is drawn as being above (solid) and the other below (broken).   The 2D projection is drawn in general form so that each strand has a unique start and end point and only two strands cross at any point (see Figures \ref{fig:braids}c, d and e).

In the 2D projection, strand positions are labelled $0$ to $n-1$ from left to right.  Using standard braid notation \cite{murasugi}, $\sigma_{i}$ represents a strand in position $i$ crossing \textbf{over} its neighbour to the right.  Similarly, $\sigma_{i}^{-1}$ represents a strand in position $i$ crossing \textbf{under} its neighbour to the right.  We will use this standard notation to represent the basic cross and twist actions of bobbin lace.   As mentioned in the introduction, bobbin lace actions are performed on four threads or two pairs of threads at a time.  A mathematically idealized thread with no thickness can be equated to a strand. If we label the pairs from left to right, the two adjacent pairs $i$ and $i+1$ correspond to the four threads in positions $2i$, $2i+1$, $2i+2$ and $2i+3$ where $i \in {0,1,2,...}$.  The cross action is represented by $\sigma_{2i+1}$ and the twist action is represented by $\sigma_{2i}^{-1}\sigma_{2i+2}^{-1}$. From this generalized description, we see that $\sigma_{x}$ will only occur for odd values of $x$ and $\sigma_{x}^{-1}$ will only occur for even values of $x$.

An \emph{alternating braid} is a braid in which each strand alternates going over and under the strands that it crosses (see Figures \ref{fig:braids}c and d).  Alternating braids are characterized by the property that the $\sigma$ generators for even positions have the opposite sign (superscript) from the $\sigma$ generators for odd positions \cite{murasugi}.  Given the generator representation for bobbin lace actions, we infer that any combination of cross and twist will result in an alternating braid.

\subsection{Systematic Explorations by Lacemakers}

Modern lacemakers have been discovering new patterns by exploring variations on traditional lace grounds.  One common approach is to take the pair working diagram from a traditional ground and look at different combinations of actions that can be performed when two pairs meet.  This combinatorial approach has been applied extensively for Rose ground (also known as Flanders ground, see Figure \ref{fig:oldpattern}c for an example) \cite{theuerkauf, viele}.  Dutch lacemaker, Pol, has developed a web-based tool for visualizing thousands of Rose ground variations \cite{pol}.

Another approach is to take a traditional ground and alter the location of the pins in the pattern.  This change does not affect the topology of the lace but has the geometric effect of changing the shape of the holes. The spaces between the threads contribute as much to the appearance of lace as the threads themselves. This approach is described in \cite{moderne} and may have been used by Kortelahti \cite{kortelahti}.

\section{Mathematical model}
\label{sec:pairtraversal}

\begin{figure*}
\begin{center}
\begin{minipage}{\textwidth}
\begin{center}
\subfigure[]{
\def\svgscale{0.2}
    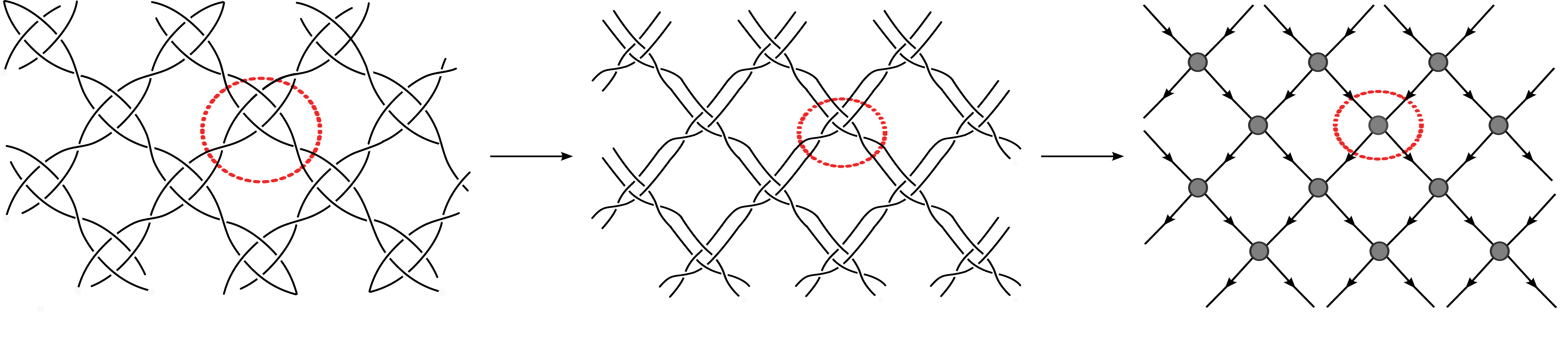
    }
\subfigure[]{
\def\svgscale{0.2}
    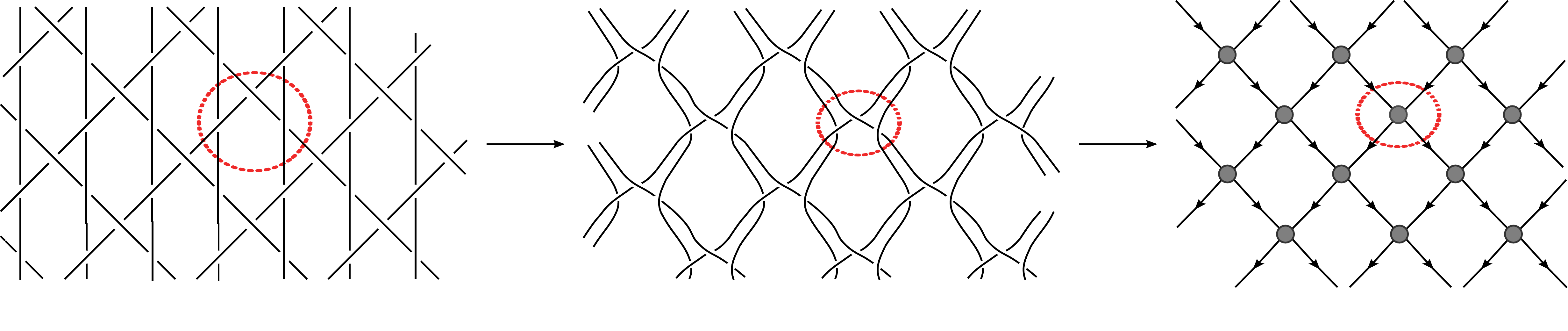
    }
\end{center}
\caption{Two example lace grounds a) Torchon Ground b) 's Gravenmoer Ground.
Left: Thread diagram of ground.
Middle: Exaggerated distance between pair crossings.
Right: Planar embedding of directed graph representing pair movement in lace ground.}
\label{fig:graph}
\end{minipage}
\end{center}
\end{figure*}

Consider four consecutive threads labelled $a, b, c, d$.  An \emph{interaction} is a sequence of actions on four threads.  The interaction begins when $b$ first crosses over $c$ and ends when any of the four threads crosses over or under a thread $x$ where $x \notin \{a,b,c,d\}$ \footnote{An interaction always begins with a cross action. The definition of an interaction should not be confused with the `closed' versus `open' methods used by lacemakers.  In the `closed' method, typically used on flat, cookie pillows, a lacemaker picks up four bobbins and always starts to braid with a cross (e.g., a half-stitch is CT). In the `open' method, used primarily on bolster style pillows, a lacemaker usually starts to braid with a twist (e.g., a half stitch is TC) but not always (e.g., the cloth stitch is CTC).    Both methods result in the same appearance of the final lace; the twist action can often be performed either at the end of one group of actions or the beginning of another without changing the end result.  Our definition of an interaction aligns more closely with the `closed' method because it has a consistent and easily defined boundary.}. Using this definition of an interaction, a lace ground can be represented as a directed graph.  The Torchon ground, shown in Figure \ref{fig:graph}a, will be used to illustrate.
The image on the left shows a diagram of the ground in which each line represents a thread and an oval is drawn around an interaction.
In the middle drawing, the distance between the interactions is exaggerated to emphasise the movement of pairs of threads from one interaction to the next.
On the right, each interaction is reduced to a dot - a vertex in the directed graph.
A pair of threads moving between two interactions is represented by a single line with an arrow indicating the direction of movement - an arc in the directed graph.
The resulting graph resembles the pair working diagram shown in Figure \ref{fig:oldpattern}b.

A bobbin lace ground can be represented as a pair $(\Pi(G^{\infty}), \zeta)$.  The first element, $\Pi(G^{\infty})$, which we shall call the \emph{ground embedding}, is an unbounded, directed graph, $G^{\infty}=(V,A)$, embedded in the plane.  The vertex set $V$ consists of all interactions and an arc $(u,v) \in A$ is a pair of thread segments travelling from interaction $u$ to interaction $v$. Usually, the graph $G$ is understood and we will shorten $\Pi(G^{\infty})$ to $\Pi^{\infty}$.  For Torchon ground, a representative subset of $\Pi^{\infty}$ is shown on the far right of Figure \ref{fig:graph}a.
The second element, $\zeta$, is a function from vertices to action sequences i.e., $\zeta : V \to \{C, T, p\}^{*}$.  For example, the action sequence for all Torchon ground interactions is $\zeta(v) = CTpCT$ for each $v \in V$.

Figure \ref{fig:graph}b shows another lace ground called \emph{'s Gravenmoer}. This second example has the same ground embedding $\Pi^{\infty}$ as Torchon but a different sequence of actions is performed at each interaction resulting in a significant difference in appearance.  Here $\zeta(v) = CTp$.

Lacemakers have already given some attention to the systematic analysis of $\zeta$.  In this paper, we focus on the motion of pairs of threads represented by the ground embedding $\Pi^{\infty}$.

In order to produce workable lace, the directed graph and its embedding must possess the five fundamental properties outlined below:

1) \textbf{2-Regular:} Two pairs of threads enter an interaction and two pairs leave making $G^{\infty}$ a directed graph of in-degree 2 and out-degree 2; such graphs are known as a 2-regular digraphs.

2) \textbf{Periodic:} A bobbin lace ground is a periodic tiling of the plane which, according to tiling theory, can be represented by a period parallelogram \cite{tilingsbook}.  Consequently, $\Pi^{\infty}$, the planar embedding of a periodic infinite graph, can be more succinctly represented by $\Pi^{\tau}$, a toroidal embedding of the directed graph $G^{\tau}=(V^{\tau}, A^{\tau})$ in which vertices $V^{\tau} \subset V$ are the set of interactions within the bounds of the period parallelogram (see Figure \ref{fig:tiling}b).

\begin{figure*}
\begin{center}
\begin{minipage}{\textwidth}
\begin{center}
\subfigure[]{
\resizebox*{!}{3.2cm}{\includegraphics{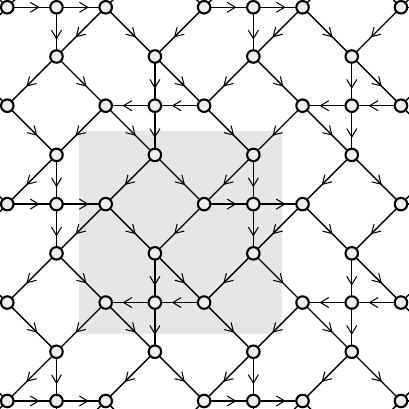}}}%
\hspace{1cm}
\subfigure[]{
\resizebox*{!}{3.2cm}{\includegraphics{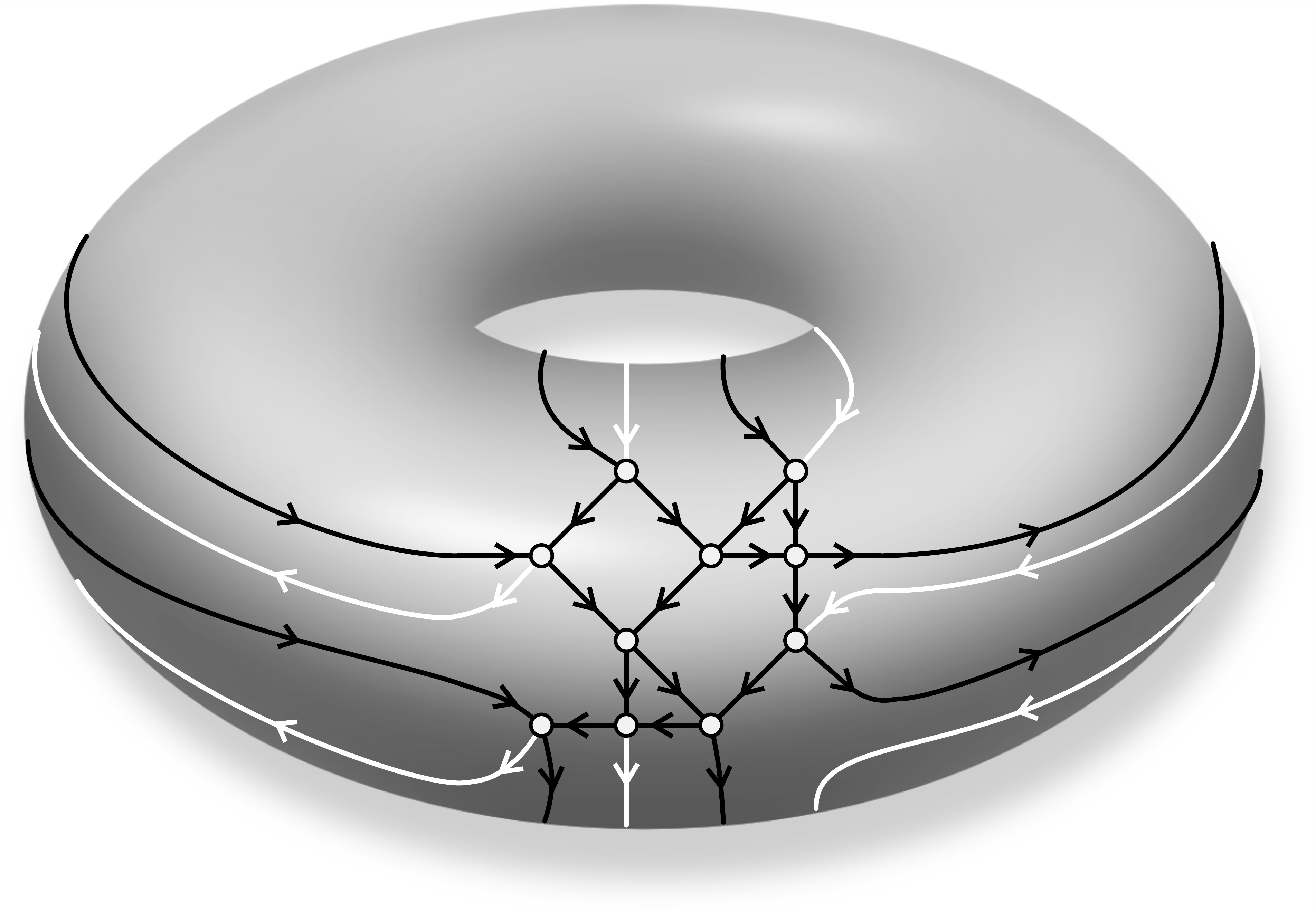}}}%
\end{center}
\caption{a) Periodic tiling of an indefinitely large region.  b) Toroidal representation of period parallelogram.}
\label{fig:tiling}
\end{minipage}
\end{center}
\end{figure*}

3) \textbf{Connected:} To ensure that threads in the lace hang together, the undirected graph $\widetilde{G}^{\infty}$, obtained from $G^{\infty}$ by ignoring the direction of the arcs, must be connected (see Figure \ref{fig:connected}).

A \emph{cycle} is a path which starts and ends at the same vertex and for which each vertex in the path, except for the first and last, appears only once.  A cycle embedded on surface $S$ is \emph{contractible} if cutting along the cycle disconnects $S$ into two parts: a topological disk and its complement. On the plane, all cyclical cuts disconnect the surface and hence are contractible.  On a surface of higher genus, this is not always true.  For example, cutting a torus along a meridian circle $M$ results in a single cylindrical surface.  After cutting $S$ along a \emph{non-contractible} cycle such as $M$, all pairs of points on the surface are still mutually reachable.

For a lace ground to be connected, $\widetilde{\Pi}^{\tau}$, the toroidal embedding of the undirected graph $\widetilde{G}^{\tau}$, must be connected \emph{and} must contain at least one non-contractible cycle (see Figure \ref{fig:torus}).  The latter property is true if $\widetilde{G}^{\tau}$ has a minimum genus of one \cite{gibbons}.

\begin{figure*}
\begin{center}
\begin{minipage}{\textwidth}
\begin{center}
\subfigure[]{
\resizebox*{!}{2.8cm}{\includegraphics{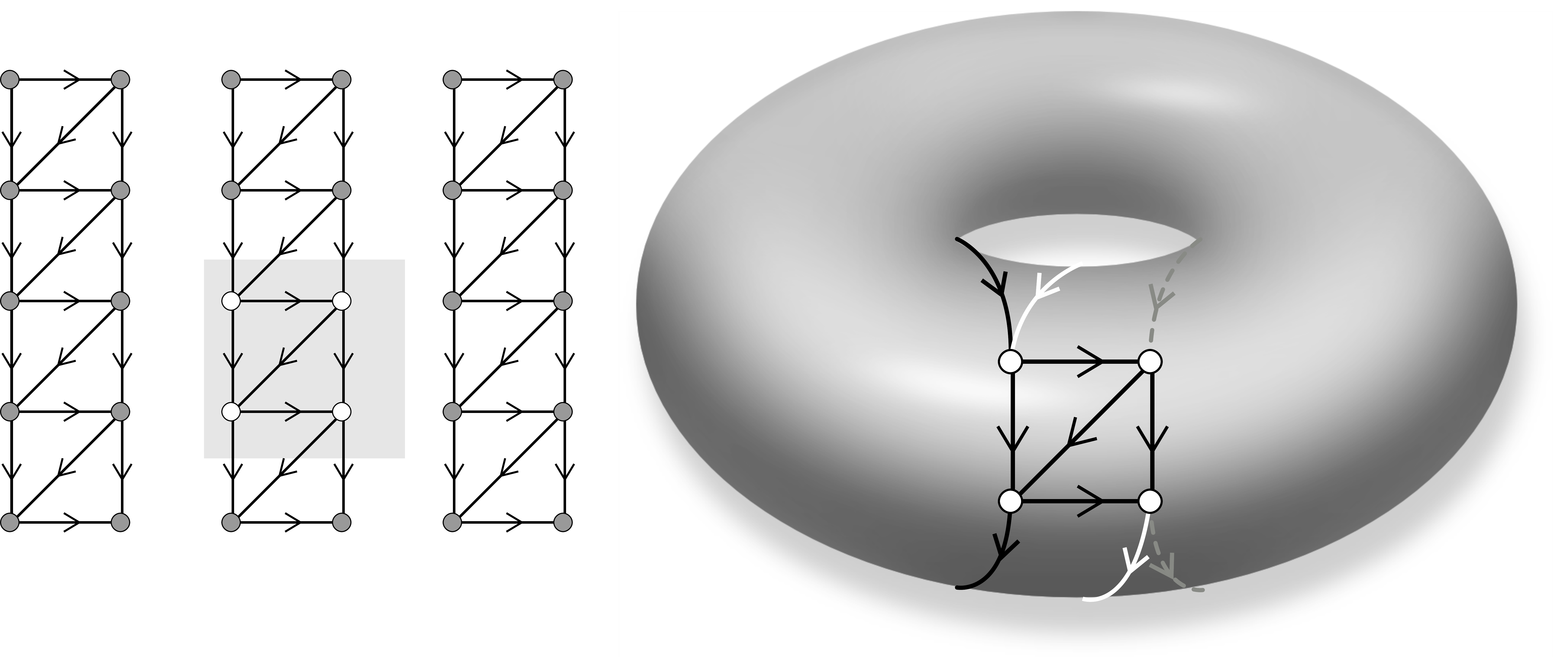}}}%
\hspace{0.5cm}
\subfigure[]{
\resizebox*{!}{2.8cm}{\includegraphics{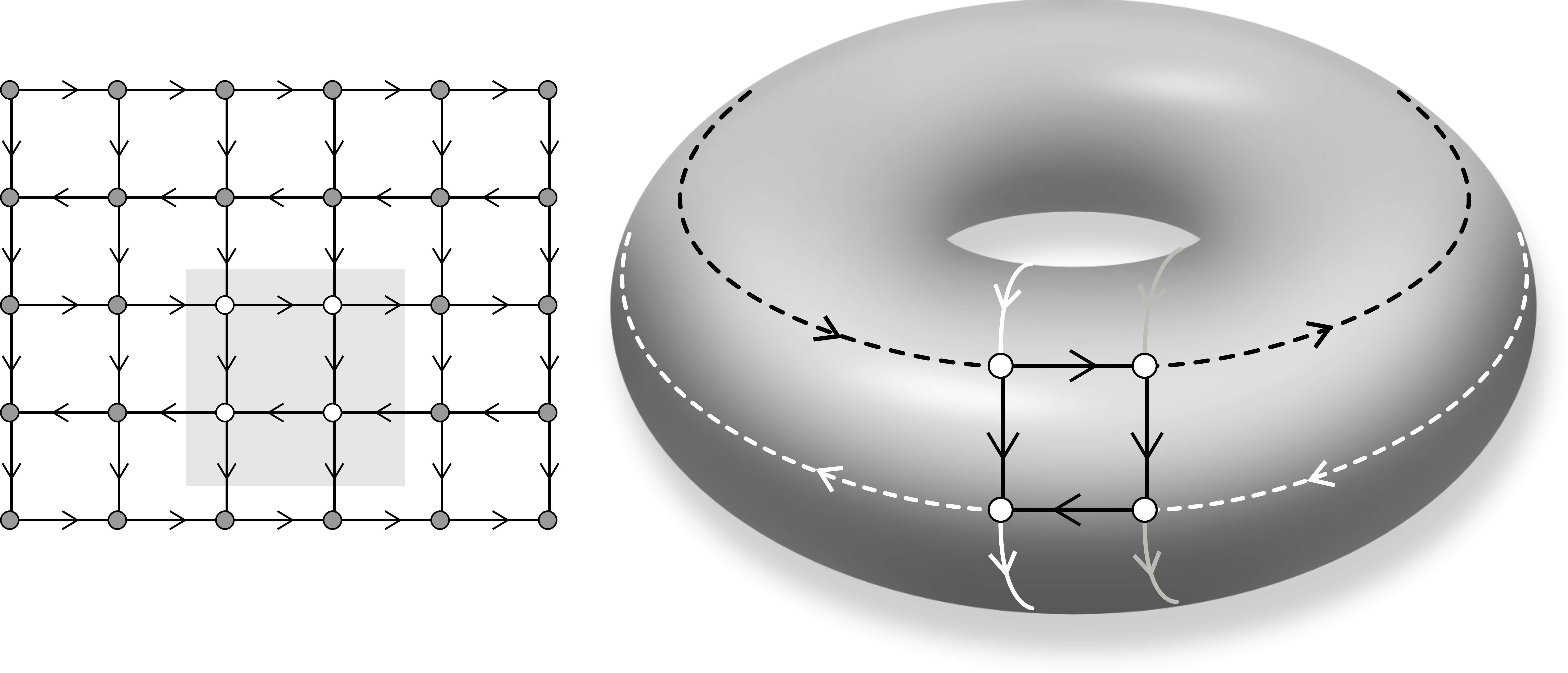}}}%
\end{center}
\caption{a) Connected on torus.  Disconnected on plane. b) Connected on torus and plane.}
\label{fig:connected}
\end{minipage}
\end{center}
\end{figure*}

4) \textbf{Monotonic:} As mentioned in Section \ref{sec:braids}, the definition of a mathematical braid requires that all strands are simultaneously monotonic with respect to the vertical axis. Bobbin lace grounds are topological braids and, by continuous deformation, can be arranged to meet this monotonic condition. However, in actual worked lace pieces, while there is a preferred downward (meridional) direction, threads may travel horizontally and even upward for short distances.  We would like a definition of monotonicity that can be tested without deforming the lace.  Another way to describe the monotonic condition of a braid is to say that sequential crossings will occur at increasingly lower vertical positions.  Starting at the top end of a strand and following it to its bottom end, we can assign a rank to each crossing based on the order in which it is encountered.    This provides a partial order for the crossings which, for a braid, must be consistent across all strands.  For example, in Figure \ref{fig:braids}b, tracing the first strand gives the order $A \prec B \prec C$ while tracing the second strand yields the order $C \prec B \prec A$.  The contradiction between these two orders indicates that the diagram does not represent a braid.

A \emph{directed cycle} is defined in the same manner as a cycle but the path must follow the direction of the arcs.  A directed graph is \emph{acyclic} if it does not contain any directed cycles.

From graph theory, we know that a directed graph representing a partial ordering (such as a Hasse diagram) is acyclic \cite{grimaldi}.  Further, when a toroidal embedding is created from the period parallelogram of a planar graph embedding, directed cycles in the plane become contractible directed cycles in the torus \cite{grishanovPart2}.  Applying this to the ground embedding, we see that $G^{\infty}$ must be acyclic and $\Pi^{\tau}$ must not have contractible directed cycles.

5) \textbf{Conserved:} Loose ends, caused by cutting threads or adding new ones, are undesirable because they inhibit the speed of working, can fray or stick out in an unslightly manner and, most importantly, degrade the strength of the fabric\footnote{One exception to the conservation of threads is the gimp thread which is a decoration and does not typically contribute to the structure of the lace.  Even so, adding and removing gimp threads at each repeat is quite tedious and to be avoided.}.  For our model, this means that, once started, a rectangular patch of width $w$ can be worked to any length without the addition or removal of threads (assuming that sufficient thread has been wound around the bobbins). We refer to this ability to extend the pattern indefinitely in the direction of monotonicity using a fixed set of threads as the \emph{conservation of threads}.  If we draw any vertical line on the pattern (that is, a line in the direction of monotonicity) then, in order for the threads to be conserved, the number of threads crossing the line from left-to-right must equal the number of threads crossing in the opposite direction.  We will formalize and prove this statement in Section \ref{sec:theorem}.

\subsection{Conservation Theorem}
\label{sec:theorem}

In order to state and prove the theorem of thread conservation, we shall first prove two auxiliary lemmas on which the proof of our theorem depends.

\begin{figure*}
\begin{center}
\subfigure[]{
\resizebox*{!}{2.5cm}{\includegraphics{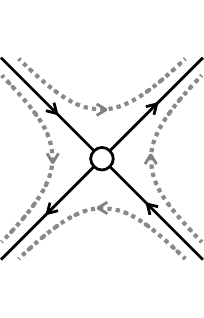}}}%
\hspace{0.5cm}
\subfigure[]{
\resizebox*{!}{2.5cm}{\includegraphics{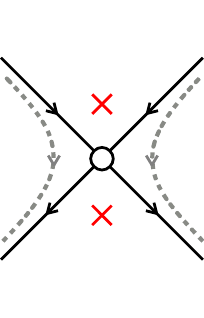}}}%
\hspace{0.5cm}
\subfigure[]{
\resizebox*{!}{2.5cm}{\includegraphics{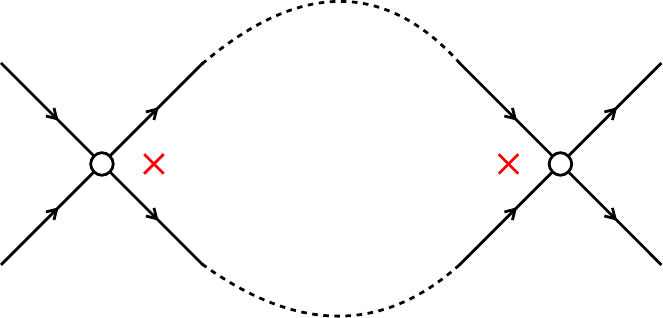}}}%
\caption{Arc arrangements at a vertex: a) Alternating and b) Consecutive. c) Incoming and outgoing blocking vertices on a face boundary.}
\label{fig:alt_cons}
\end{center}
\end{figure*}

A \emph{directed circuit} is a directed path that starts and ends at the same vertex and each arc in the path is unique. In contrast to a directed cycle, a vertex in a directed circuit may be visited more than once.
Arcs in an embedding of a 2-regular digraph can be arranged in one of two possible ways around a vertex: either \emph{rotationally alternating} in which arcs alternate between incoming and outgoing directions or \emph{rotationally consecutive} with arcs in the order incoming, incoming, outgoing, outgoing (see Figures \ref{fig:alt_cons}a \& \ref{fig:alt_cons}b).  We shall refer to a vertex, with reference to its arc arrangement, as a \emph{rotationally alternating vertex} or a \emph{rotationally consecutive vertex}.

\begin{lemma}
Let $\Pi^{g}$ be a 2-regular digraph with $n$ vertices embedded on an orientable surface of genus $g$.  If $\Pi^{g}$ contains no contractible directed cycles, then it must have at least $2-2g+n$ rotationally consecutive vertices.
\label{theor:nocontract}
\end{lemma}

\begin{proof}
From the Euler characteristic, we have the well known result that an embedding on an orientable surface has $2-2g+n$ faces. We shall prove that there are at least as many rotationally consecutive vertices as faces.  Consider a face $F$ of $\Pi^{g}$. For every vertex $a$ with a rotationally alternating arc configuration, the arcs incident to $F$ at $a$ enter and exit with the same orientation. If every vertex in the boundary of $F$ has a rotationally alternating arc configuration, then the boundary is a contractible directed cycle, a condition we wish to avoid.  Therefore we may assume that $F$ has at least one vertex $c$ with a rotationally consecutive arc configuration such that the arcs incident on $F$ at $c$ are either both incoming or both outgoing.
Such a vertex $c$ prevents a directed cycle for two (of the at most four) incident faces, one blocked by its incoming edges and one by its outgoing edges.
Since an outgoing arc at $c$ is an incoming arc at the next vertex in the boundary of $F$, it follows that the two outgoing arcs at $c$ must be balanced by two incoming arcs at one or two vertices in $F$ (see Figure \ref{fig:alt_cons}c). The same balancing argument holds for two incoming arcs at $c$.  In summary, each rotationally consecutive vertex blocks two faces but each face must have at least two rotationally consecutive vertices.  Therefore there are at least as many rotationally consecutive vertices as faces.
\end{proof}

While requiring a minimum number of vertices to have consecutive incoming arcs is a necessary condition to avoid contractible directed cycles (as proven above), it is not a sufficient condition.

By applying Lemma \ref{theor:nocontract} to planar ($g = 0$) and toroidal ($g=1$) embeddings, we derive the following corollary:
\begin{corollary} A finite 2-regular digraph embedded on the plane will always have a contractible directed cycle bounding a face.
To be free of contractible directed cycles, all vertices of a 2-regular digraph embedded on the torus must be rotationally consecutive.
\label{theo:allconsecutive}
\end{corollary}

\begin{figure*}
\begin{center}
\subfigure[]{
\resizebox*{!}{3cm}{\includegraphics{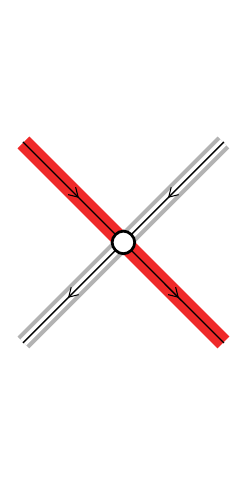}}}%
\hspace{0.5cm}
\subfigure[]{
\resizebox*{!}{3cm}{\includegraphics{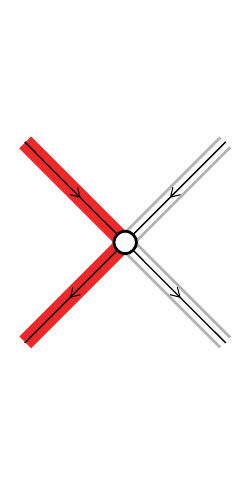}}}%
\hspace{0.5cm}
\subfigure[]{
\resizebox*{!}{3cm}{\includegraphics{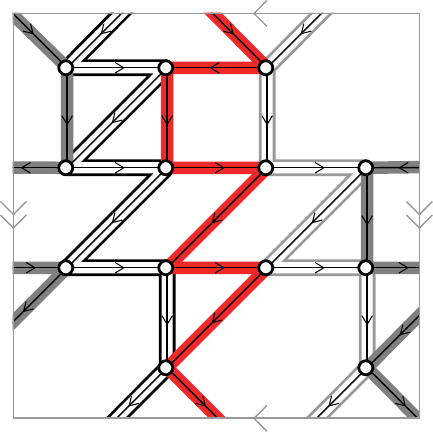}}}%
\hspace{0.5cm}
\subfigure[]{
\resizebox*{!}{3cm}{\includegraphics{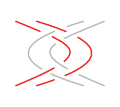}}}%
\caption{a) Transverse intersection. b) Non-transverse intersection. c) Toroidal graph partitioned into four directed circuits by a non-transversal weave. d) $\zeta(v) = CTCCTC$.}
\label{fig:crossing}
\end{center}
\end{figure*}

A \emph{transverse intersection} occurs when one path lies across another path in an embedding as illustrated in Figure \ref{fig:crossing}a.  In contrast, a \emph{non-transverse intersection} occurs when two paths meet but continue without crossing (Figure \ref{fig:crossing}b).

\begin{lemma}
Let $\Pi^{\tau}$ be a 2-regular digraph $G(V,A)$ embedded on a torus.  If $\Pi^{\tau}$ does not contain any contractible directed cycles then the arcs of $G(V,A)$ can be partitioned into a set of directed circuits such that no circuit participates in a transverse intersection.
\label{theor:partition}
\end{lemma}

\begin{proof}
Begin circuit $K$ by selecting any arc $(u,v)\in A$.  Proceed to the vertex at the far end of the arc, $v$, and add the unique, rotationally consecutive outgoing arc $(v,w) \in A$ to $K$.  The existence and uniqueness of $(v,w)$ is guaranteed by Corollary \ref{theo:allconsecutive}.  The process is repeated using vertex $w$ until $K$ returns to its initial vertex $u$ via arc $(z,u) \in A$ where $(z,u)$ is rotationally consecutive to $(u,v)$ at vertex $u$. If $K$ returns to its initial vertex $u$ via the arc $(z',u)$ which is not rotationally consecutive to $(u,v)$, continue tracing the circuit until it returns to $u$ a second time. The circuit $K$ is guaranteed to complete because there are an equal number of incoming and outgoing arcs at each vertex and a finite number of vertices.  Remove all arcs traversed by $K$ from $\Pi^{\tau}$. The resulting embedding still has an equal number of incoming and outgoing arcs and the relative rotational order of the remaining arcs is unchanged so the process can be repeated until zero arcs remain.  The resulting set of circuits will not intersect each other transversely because at each stage, the selected arcs were consecutive.
\end{proof}

Figure \ref{fig:crossing}c shows an example of the non-transversal partitioning of a graph embedding.

\begin{figure*}
\begin{center}
\subfigure[]{
 \def\svgscale{0.09}
    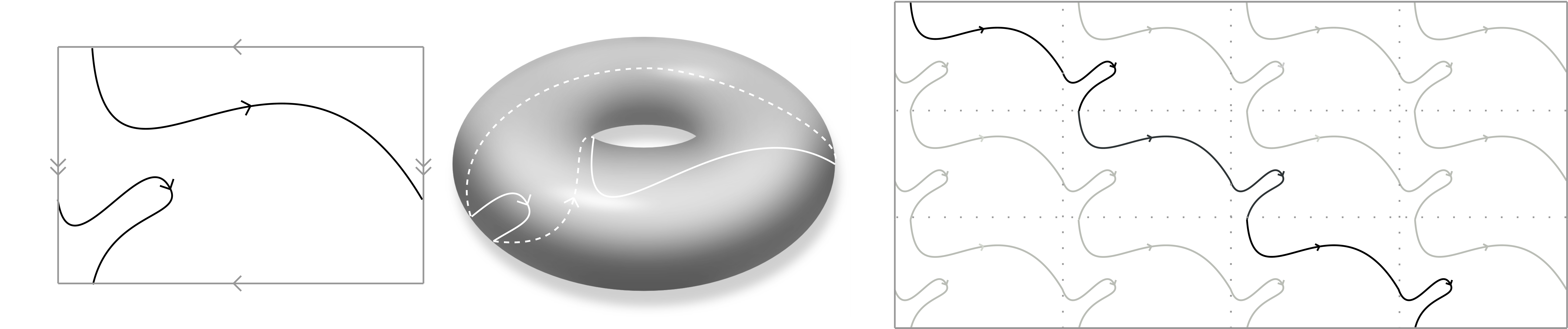 
}
\subfigure[]{
 \def\svgscale{0.09}
    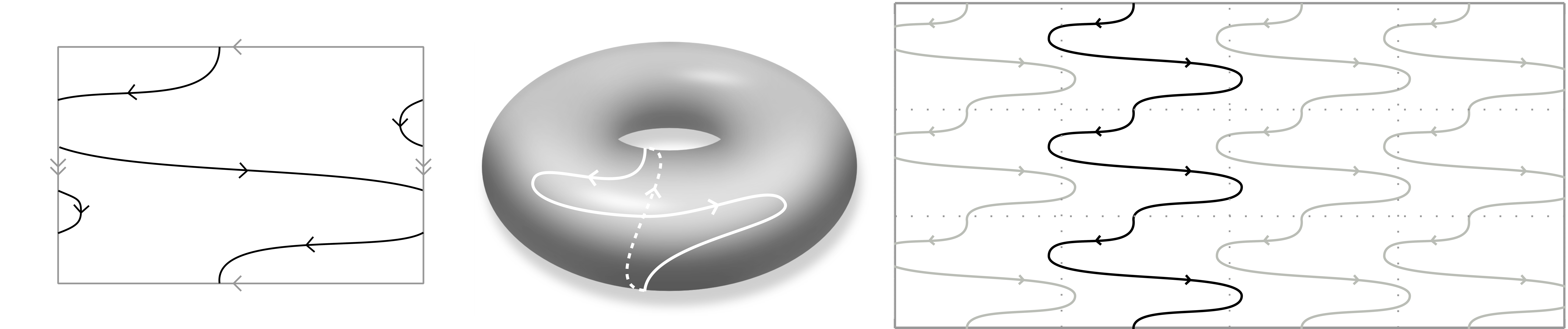 
}
\caption{Circuit on a torus. a) Circuit wraps once around torus longitudinally and once meridionally. b) Circuit wraps once around the torus meridionally, no longitudinal wrapping.}
\label{fig:wrapping}
\end{center}
\end{figure*}

\begin{theorem}
Let $\Pi^{\tau}$ be a toroidal embedding of a 2-regular digraph with no contractible directed cycles.  The embedding is thread conserving if and only if $\Pi^{\tau}$, when cut in the direction of monotonicity, has the same number of arcs crossing the cut in one direction as in the opposite direction.
\end{theorem}

\begin{proof}
Due to Lemma \ref{theor:partition}, the arcs of $\Pi^{\tau}$ can be partitioned into a set of non-transverse directed circuits, each of which is non-contractible.  By using $\zeta(v) = CTCpCTC$ for all vertices in $\Pi^{\tau}$ (see Figure \ref{fig:crossing}d), we can equate a circuit in this partition with the path followed by a pair of threads.
Using the homology group of the torus, $\mathbb{Z}\times\mathbb{Z}$, we see that a directed circuit that wraps longitudinally around the torus $L$ times will cross a meridian circle a net of $L$ times, where a crossing in the direction of the circuit is positive and a crossing in the opposite direction is negative.   As illustrated in Figure \ref{fig:wrapping}, a circuit that longitudinally wraps $L$ times  around $\Pi^{\tau}$  corresponds to a path in the associated $\Pi^{\infty}$ with a horizontal displacement between start and finish vertices of $L$ times the width of the period parallelogram.  For a constant number of threads to cover a rectangle of fixed width and indeterminate length, there must be no horizontal displacement of the path from one repeat to the next, corresponding to $L=0$.  Therefore, the circuit must cross any meridian circle (where the meridian is the direction of monotonicity) a net of zero times.  Since the circuits in the partition do not intersect each other transversely, all circuits in $\Pi^{\tau}$ will wrap around the torus the same number of times.
This result can be generalised for any $\zeta$ function by noting that, within an interaction, the threads are the manifestation of a mathematical braid which, by definition, conserves the number of strands in the braid. Therefore, replacing $\zeta(v) = CTCpCTC$ with any $\zeta(v)$ will not alter the number of threads required.
\end{proof}

\section{Application of model}
\label{sec:variations}

Now that we have a mathematical model, we shall apply it to the generation of some lace grounds.  Ideally, one would pick a number of vertices and generate all $\Pi^{\tau}$ of this size satisfying the fundamental properties specified in Section \ref{sec:pairtraversal}.  Unfortunately, such a general approach requires solving several intractable problems, the scale of which would limit our solution to very small numbers of vertices.  As a result, we tried a simpler approach that did not find all possible solutions but did generate an interesting set of results and revealed some novel ground patterns.

\begin{figure*}
\begin{center}
 \def\svgscale{0.9}
    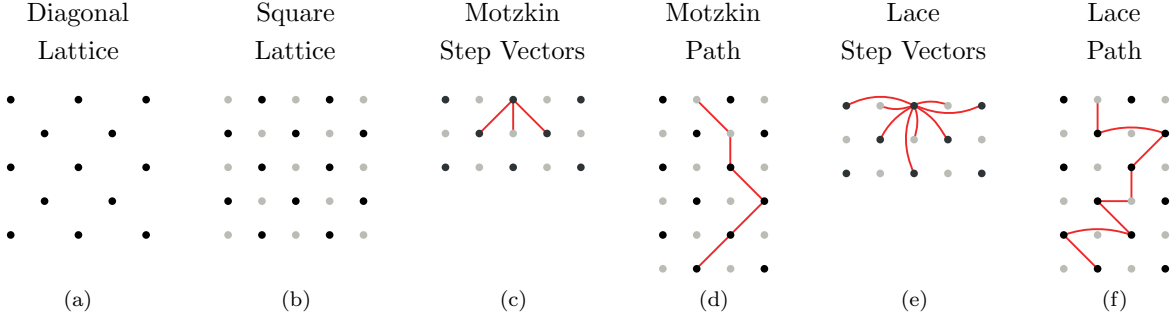
\caption{a) Typical lace grid (Torchon $45^{\circ}$ example shown). b) Square lattice created from lace grid. c) Step vectors for Motzkin paths. d) An example Motzkin path with 5 steps. e) Step vectors for lace paths. f) An example lace path with 8 steps.}
\label{fig:stepvectors}
\end{center}
\end{figure*}

A lace pattern is designed on top of a diagonal lattice, referred to in lace literature as a grid (see, for example, \cite{nottingham2}).  From Lemma \ref{theor:partition}, we know that a lace ground embedding can be partitioned into a set of circuits that do not intersect transversely.   Putting these two ideas together, we decided to look at lace ground embeddings constructed from a non-transversal weave of lattice paths.  A \emph{lattice path} is a sequence of line segments that travel from lattice point to lattice point using a finite list of allowed motions.  For example, as shown in Figure \ref{fig:stepvectors}c, the well known Motzkin paths have the allowed steps $\swarrow$, $\downarrow$ and $\searrow$\footnote{To align with the top to bottom way lace is designed, we have rotated the steps $90^{\circ}$ clockwise from their traditional left to right orientation} \cite{donaghey}.  These allowed motions can be expressed more precisely as a set of step vectors $\mathfrak{M}=\{\step{-1,1}, \step{0,1}, \step{1,1}\}$ in which $\step{x,y}$ indicates a step of $x$ units in the horizontal direction and $y$ units in the vertical direction.  For lace grounds, we chose a larger range of step vectors: $\mathfrak{L}=$ $\{\step{-1,1}$, $\step{0,1}$, $\step{1,1}$, $\step{0,2}$, $\step{1,0}$, $\step{-1,0}$, $\step{2,0}$, $\step{-2,0}\}$ (see Figure \ref{fig:stepvectors}e).  The paths produced from these step vectors will be referred to as \emph{lace paths}. A lace path starts at position $(x,y)$, ends at position $(x,y+n)$ and may travel left and/or right of the vertical line connecting these two points.  In addition, a lace path can not have consecutive horizontal steps (i.e., $\step{2,0}$ can not be followed by $\step{1,0}$, $\step{-1,0}$, $\step{2,0}$ or $\step{-2,0}$).  If a lace path has two or more consecutive horizontal steps, it can only be combined with another lace path via a transverse intersection.  As described earlier, all intersections in the solution must be non-transverse.

\begin{figure*}
\begin{center}
\resizebox*{!}{4cm}{\includegraphics{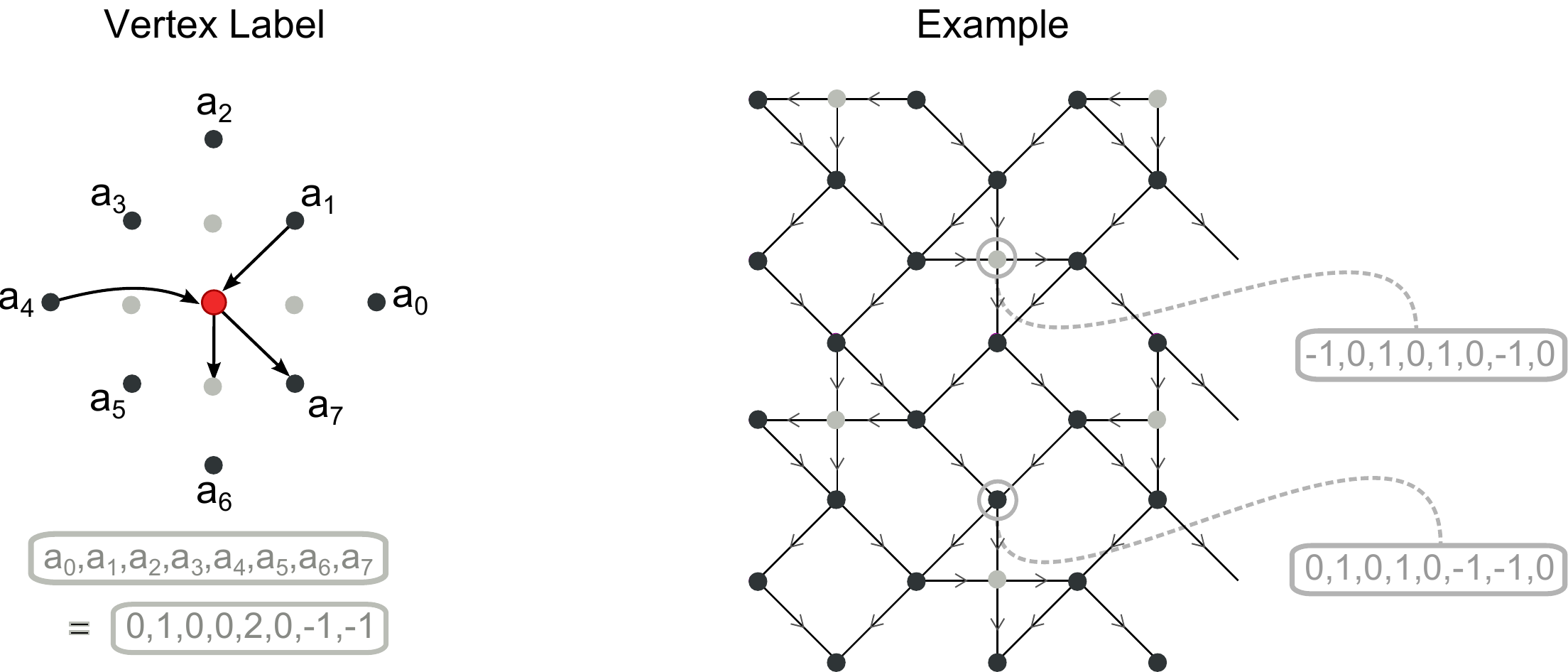}}
\caption{Left: Labelling scheme for arcs around a vertex. The absolute value of the number in the label indicates the size of step taken. The sign of the value indicates incoming (+) versus outgoing (-) direction.  Right: Example solution with two vertices labelled.}
\label{fig:labelling}
\end{center}
\end{figure*}

We wish to identify patterns invariant under the transformations of translation, horizontal or vertical reflection and $180^{\circ}$ rotation. To do this, we define a canonical form of the embedding and discard any solutions not in canonical form.  Each vertex in $\Pi^{\tau}$ is labelled using the rotational order of its arcs.  The label has the form $(a_0, a_1, a_2, a_3, a_4, a_5, a_6, a_7)$ where $a_i$ encodes the direction and length of an arc travelling to or from relative position $i$ (incoming arcs are positive, outgoing arcs are negative, length is the floor of the distance between arc endpoints). See Figure \ref{fig:labelling} for an example.  An identifier is constructed by stringing together vertex labels in row by row order. The lexicographically least form of this identifer is found by translating the period parallelogram so that the vertex with the lexicographically least label is in the top left corner.
The process is repeated for the horizontal and vertical reflection and $180^{\circ}$ rotation of $\Pi^{\tau}$ (Note: In these last two cases, the direction of each arc is reversed). The canonical form of $\Pi^{\tau}$ is the lexicographically least form of all of these identifiers.

\section{Algorithm for Generation and Enumeration}

Backtracking is an algorithm often used for exhaustive enumeration.  For example, Ramanath and Walsh used backtracking to enumerate 2-regular directed graphs ~\cite{ramanath}.
In backtracking, a tree is created in which internal nodes represent partial solutions and leaves represent complete solutions.  Each child node extends the partial solution of its parent by adding one additional move or action.  As each node is added to the tree, the partial solution is evaluated to determine if it has the  potential to complete successfully. If a partial solution can not be completed, the associated branch is terminated, a condition commonly referred to as early branch termination.
Backtracking often yields a huge improvement over the brute force approach of evaluating all possible combinations.  The key to improving performance is to keep the tree as small as possible by minimizing the number and length of the branches.  This can be done through the choice of good early termination conditions and a small set of configurations for each branching point.

In our application, internal nodes of the backtracking tree contain a partially completed toroidal ground embedding, $\Pi^{\tau}$.  At each child node the arcs of a lace path are added to the solution.  Early branch termination tests check that each vertex has a maximum degree of 2-in,2-out, arcs cross only at vertices, and internal vertex labels are lexicographically greater than or equal to the top left vertex label.  When an embedding is successfully completed, it is further tested for connectedness and an exact degree of 2-in,2-out.

In Algorithm \ref{alg:backtrack}, we represent $\Pi^{\tau}$ as a 2 dimensional array of vertex labels.  The array corresponds to a rectangular $n\times m$ grid representing the period parallelogram of the ground with each cell representing a square lattice point.  When the algorithm produces a complete solution, it may contain vertices with an empty set of arcs.  These vertices correspond to unused lattice positions and can be ignored.

\begin{algorithm2e}
\DontPrintSemicolon
\small{
\KwIn {$n$, Number of rows of in lattice}
\KwIn {$m$, Number of columns in lattice}
\Begin{
    $L$ $\leftarrow$ Generate an exhaustive list of lace paths from $(0,0)$ to $(n,0)$ \;
    Initialize $\Pi^{\tau}$ with isolated vertices \;
    \textsc{back}$(\Pi^{\tau}, L, 0)$ \;
}}
\caption{\textsc{enumerate}$(n,m)$}
\label{alg:backtrack}
\end{algorithm2e}

\begin{algorithm2e}
\DontPrintSemicolon
\NoCaptionOfAlgo
\small{
\KwIn{$\Pi^{\tau}$, An $n\times m$ ground embedding}
\KwIn{$L$, Set of lace path configurations}
\KwIn{$index$, Index of a column in ground embedding}
\Begin{
    \If{$index = n$}{
        \If{\textsc{validEmbedding}$(\Pi^{\tau})$} {
            Save ground embedding to file system \;
            Update total count \;
       }
    }\Else{
        \tcc{Vertex $(0,index)$ not used}
        \textsc{back}$(\Pi^{\tau}, L, index+1)$ \;

        \For{$l \in L$} {
            $cloneA \leftarrow$ clone of $\Pi^{\tau}$ \;
            \If{\textsc{addPath}$(cloneA, l, index)$} {
                \tcc{1 path}
                \textsc{back}$(cloneA, L, index+1)$ \;
                \For{$ll \in L$} {
                    $cloneB \leftarrow$ clone of $cloneA$ \;
                    \If{\textsc{addPath}$(cloneB, ll, index)$} {
                        \tcc{2 paths}
                        \textsc{back}$(cloneB, L, index+1)$ \;
                    }
                }
            }
        }
    }
}}
\caption{\textsc{back}$(\Pi^{\tau}, L, index)$}
\end{algorithm2e}

\begin{algorithm2e}
\DontPrintSemicolon
\NoCaptionOfAlgo
\small{
\KwData{$\Pi^{\tau}$, An $n\times m$ ground embedding}
\KwData{$l$, List of step vectors describing a lattice path}
\KwData{$index$, Index of a column in ground embedding}
\KwResult{ True if path added successfully}
\Begin{
    $v \leftarrow$ vertex at position $(0, index)$ \;
    \For{$step \in l$}{
        $arc \leftarrow$ Create arc $(v,w)$ using $step$ \;
        Add $arc$ to outgoing arc set of $v$ \;
        \If{not \textsc{validVertex}$(v, \Pi^{\tau})$}{
            \Return{false}
        }
        Add $arc$ to incoming arc set of $w$ \;
        \If{not \textsc{validVertex}$(w, \Pi^{\tau})$} {
            \Return{false}
        }
        $v \leftarrow w$ \;
    }
    \Return{true}
}}
\caption{\textsc{addPath}$(\Pi^{\tau}, l, index)$}
\end{algorithm2e}

\begin{algorithm2e}
\DontPrintSemicolon
\NoCaptionOfAlgo
\small{
\KwData{$v$, A vertex in $\Pi^{\tau}$}
\KwData{$\Pi^{\tau}$, An $n\times m$ ground embedding}
\KwResult{True if $v$ meets all intermediate solvability tests}
\Begin{
    Ensure vertex $v$ has at most 2 in and 2 out arcs \;
    Ensure arcs of $v$ do not cross any existing arcs \;
    Ensure arc list of $v$ is lexicographically less than or equal to arc list of vertex $(0,0)$ of $\Pi^{\tau}$ \;
    Repeat previous step for horizontal and vertical reflection and $180^{\circ}$ rotation of $v$ \;
    \Return{true if all conditions are met}
}}
\caption{\textsc{validVertex}$(v, \Pi^{\tau})$}
\end{algorithm2e}

\begin{algorithm2e}
\DontPrintSemicolon
\NoCaptionOfAlgo
\small{
\KwData{$\Pi^{\tau},$ An $n\times m$ ground embedding}
\KwResult{True if $\Pi^{\tau}$ meets all criteria}
\Begin{
    Ensure all non-isolated vertices have exactly 2 incoming and 2 outgoing arcs \;
    Ensure $\Pi^{\tau}$ is connected \;
    \Return{true if all conditions are met}
}}
\caption{\textsc{validEmbedding}$(\Pi^{\tau})$}
\end{algorithm2e}

\section{Results}
\label{sec:results}
In Section \ref{sec:output} we demonstrate that the model produces lace grounds that can be worked alone or in combination.   In Section \ref{sec:compare} we compare the results obtained from our algorithm to traditional lace grounds.

\subsection{Output from Algorithm}
\label{sec:output}
The enumeration results for lace paths of vertical height $n$, shown in Table \ref{table:lacepaths}, were determined algorithmically.  The corresponding lace ground results, shown in Table \ref{table:resulttable}, are accumulative.  For example, the $3\times 3$ enumeration includes the $1 \times 1$, $1 \times 3$ and $3 \times 1$ ground embeddings.

The backtracking algorithm was implemented in Java. Since each branch of the backtracking tree can be evaluated independently, the fixed thread pool service of Java's concurrent Executor class was used to process branches in parallel.  The algorithm was executed on a 3.40 GHz machine with eight Intel i-4770 cores, 16 GB RAM and a 64 bit Windows Operating System.  The size of the period parallelogram explored using this approach was limited by the performance of the algorithm.  For example, the $4\times 4$ enumeration completed in 431 hours and created a tree with approximately $7\times 10^{8}$ nodes.

\begin{table}
\caption{Enumeration of lace paths of height $n$.}
\label{table:lacepaths}
\begin{center}
{\begin{tabular}{@{}rrrrrr}
  \hline
  $n=$         & 1  & 2   & 3     & 4      & 5            \\
  \hline
  path count = & 3  & 39  & 498   & 6667   & 91833        \\
  \hline
\end{tabular}}
\end{center}
\end{table}

\begin{table}
\begin{center}
\begin{threeparttable}
\caption{Enumeration of $\Pi^{\tau}$ using backtracking algorithm.}
\label{table:resulttable}
{
\begin{tabular}{@{}crrrrrc}
  \hline
  $n\backslash m$ & 1  & 2   & 3     & 4      & 5             & \\ \hline
  1               & 1  & 2   & 2     & 4      & 4             & \multirow{5}{*} {\def\svgscale{0.5}
\begingroup%
  \makeatletter%
  \providecommand\color[2][]{%
    \errmessage{(Inkscape) Color is used for the text in Inkscape, but the package 'color.sty' is not loaded}%
    \renewcommand\color[2][]{}%
  }%
  \providecommand\transparent[1]{%
    \errmessage{(Inkscape) Transparency is used (non-zero) for the text in Inkscape, but the package 'transparent.sty' is not loaded}%
    \renewcommand\transparent[1]{}%
  }%
  \providecommand\rotatebox[2]{#2}%
  \ifx\svgwidth\undefined%
    \setlength{\unitlength}{58.92519531bp}%
    \ifx\svgscale\undefined%
      \relax%
    \else%
      \setlength{\unitlength}{\unitlength * \real{\svgscale}}%
    \fi%
  \else%
    \setlength{\unitlength}{\svgwidth}%
  \fi%
  \global\let\svgwidth\undefined%
  \global\let\svgscale\undefined%
  \makeatother%
  \begin{picture}(1,1.05952816)%
    \put(0,0){\includegraphics[width=\unitlength]{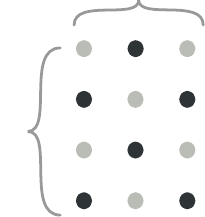}}%
    \put(0.07355227,0.10626131){\color[rgb]{0,0,0}\rotatebox{90}{\makebox(0,0)[lb]{\smash{{\footnotesize $n=4$}}}}}%
    \put(0.35152975,1.12178019){\color[rgb]{0,0,0}\makebox(0,0)[lb]{\smash{{\footnotesize $m = 3$}}}}%
  \end{picture}%
\endgroup%
} \\
  2               & 4  & 12  & 31    & 126    & 542           &\\
  3               & 6  & 31  & 274   & 3,527   & 53,196     &\\
  4               & 27 & 188 & 4029  & 134,012 & $\geq 78,061^{\rm a}$   &\\
  5               & 82 & 937 & 64,050 & $\geq 124,100^{\rm b}$    &         &\\
  \hline
\end{tabular}
}
\begin{tablenotes}
    \small
    \item $^{\rm a}$ Terminated before completion, tree contained $1.48 \times 10^{11}$ nodes
    \item $^{\rm b}$ Terminated before completion, tree contained $1.26 \times 10^{11}$ nodes
\end{tablenotes}
\end{threeparttable}
\end{center}
\end{table}

\begin{figure*}
\begin{center}
\begin{minipage}{\textwidth}
\subfigure[]{
    \resizebox*{!}{3cm}{\includegraphics{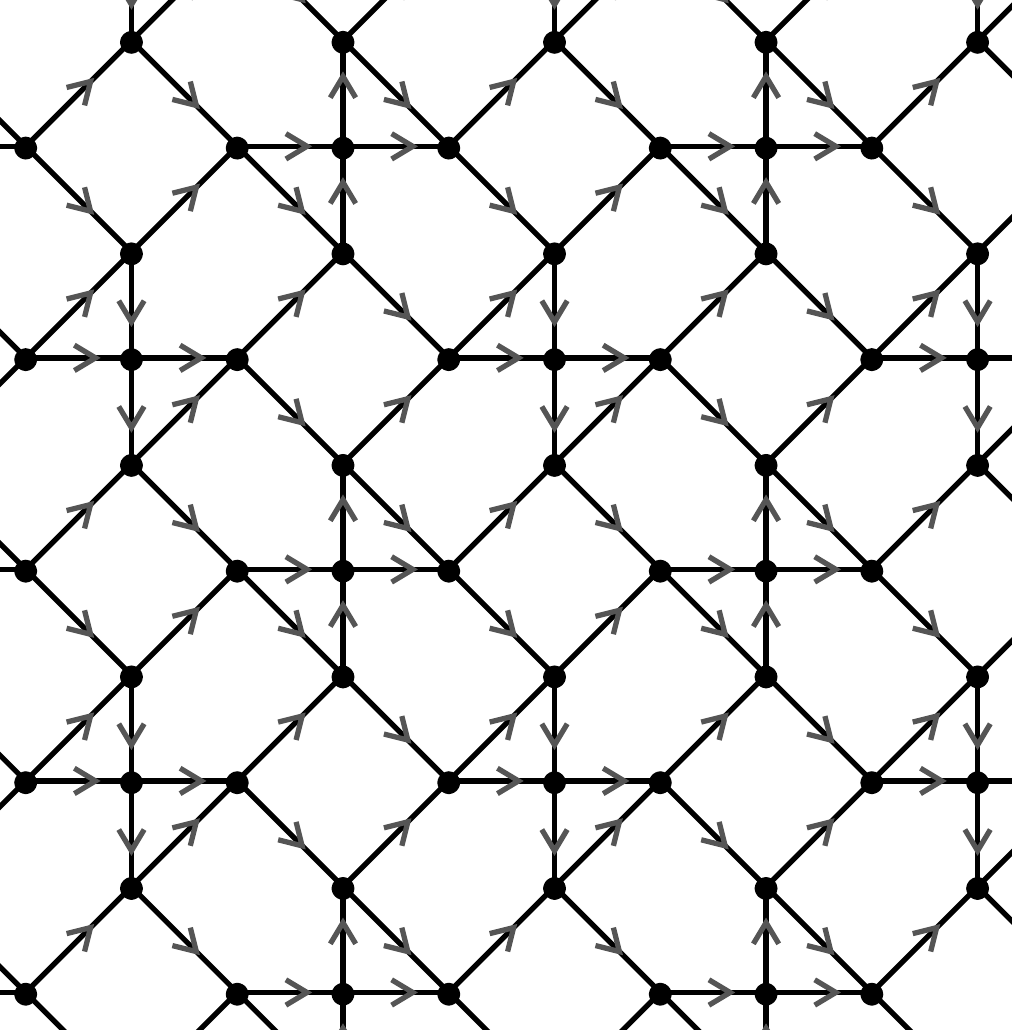}}
}
\subfigure[]{
    \resizebox*{!}{3cm}{\includegraphics{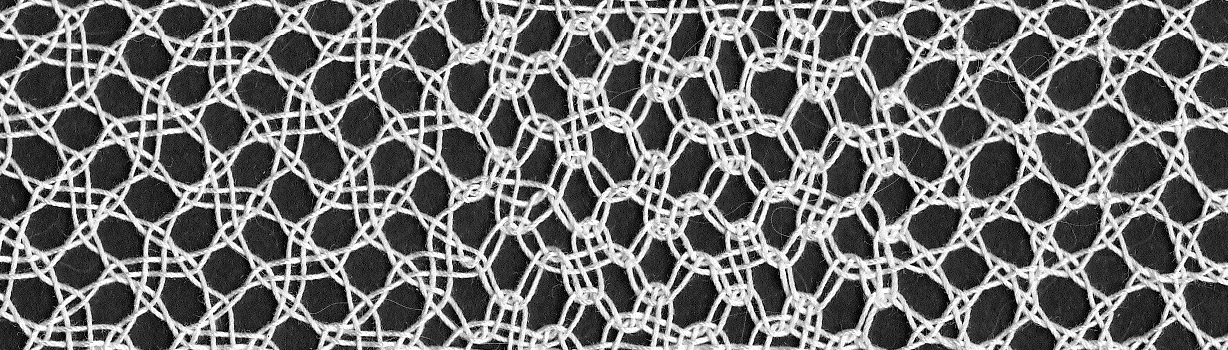}}
}
\\
\subfigure[]{
    \resizebox*{!}{3cm}{\includegraphics{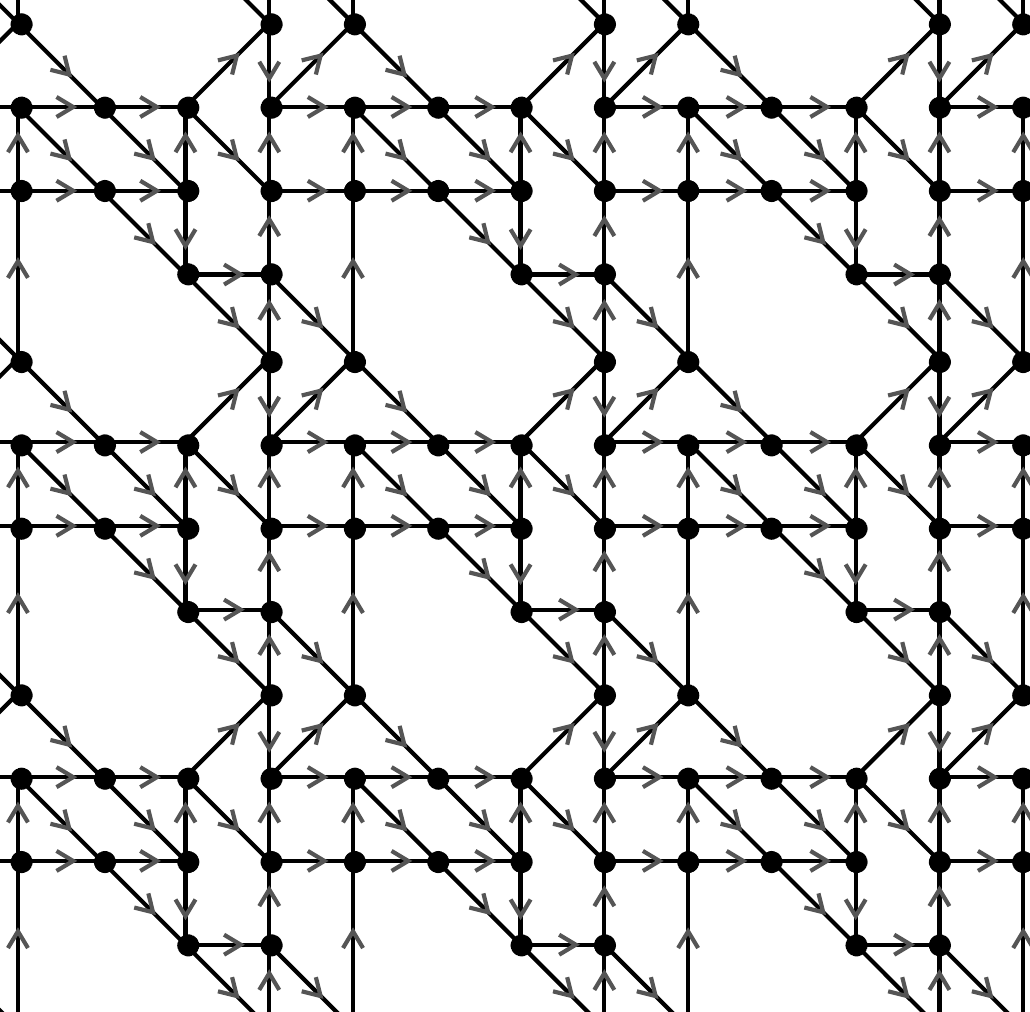}}
}
\subfigure[]{
    \resizebox*{!}{3cm}{\includegraphics{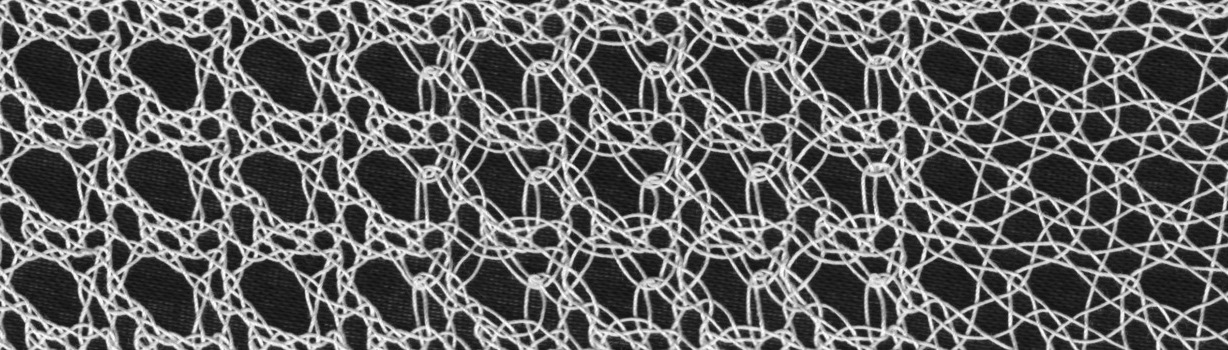}}
}
\\
\subfigure[]{
    \resizebox*{!}{3cm}{\includegraphics{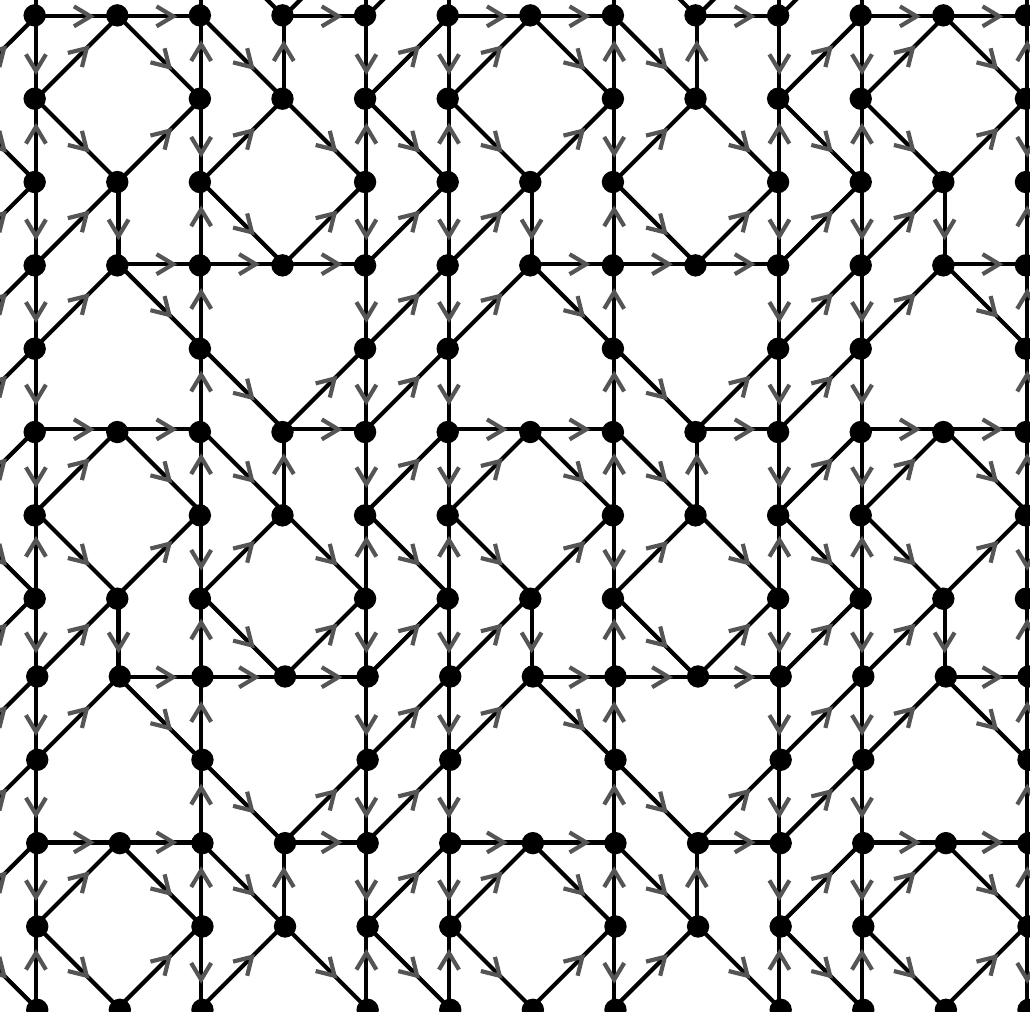}}
}
\subfigure[]{
    \resizebox*{!}{3cm}{\includegraphics{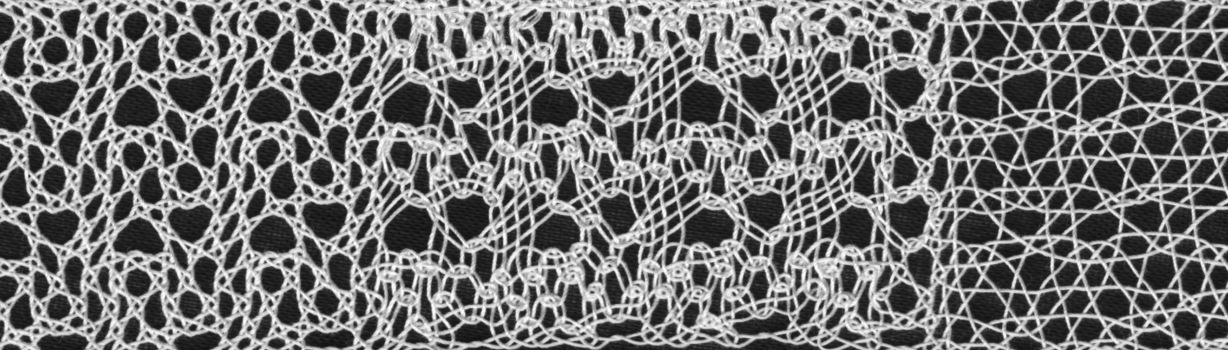}}
}
\caption{New lace grounds discovered via algorithm.  Each ground embedding was worked using three different $\zeta(v)$ functions.  As far as the authors are aware, these patterns have not been previously documented. The patterns and finished lace are rotated $90^{\circ}$ counter clockwise.}
\label{fig:worked}
\end{minipage}
\end{center}
\end{figure*}

Several samples of the solutions in Table \ref{table:resulttable} have been worked using a variety of action sequences.  Three of these  are shown in Figure \ref{fig:worked}.  The sample on the top demonstrates four-fold rotational symmetry and belongs to the orbifold wallpaper group $4*2$ \cite{conway}.  While most traditional lace grounds exhibit some form of reflection, rotational symmetry is fairly uncommon.  The second and third examples show unusual hole shapes and a complex flow of threads.

In Figure \ref{fig:bookmark2}, we present an original bookmark pattern using ground embeddings identified by our algorithm.  The bookmark demonstrates that the new grounds can be used together.

\subsection{Comparison to traditional lace grounds}
\label{sec:compare}

There are a number of excellent reference catalogues documenting the lace grounds found in traditional patterns \cite{cook, system, viele}.
A detailed comparison between our results and the reference catalogues is challenging because the catalogues do not always give a full description of how pairs of threads travel through a ground.  In many cases, this information must be derived from thread diagrams or pricking diagrams which requires quite a bit of time.  However, we can draw some general conclusions.

\begin{figure*}
\begin{center}
\begin{minipage}{\textwidth}
\begin{center}
\subfigure[]{
\resizebox*{!}{2.2cm}{\includegraphics{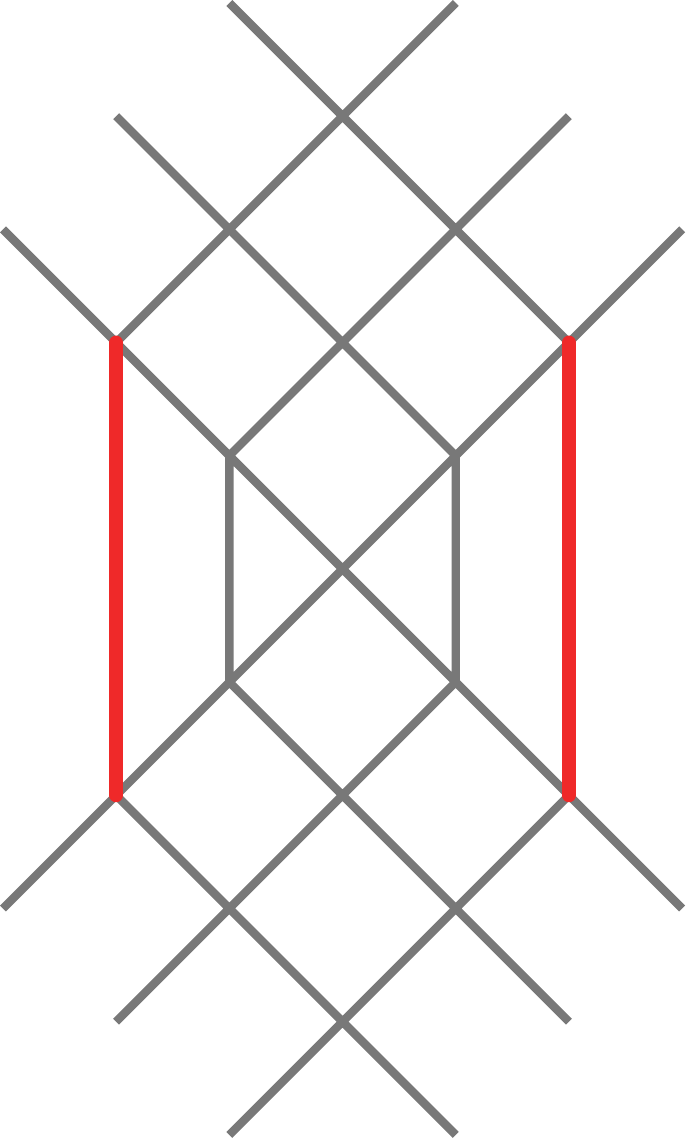}}}%
\hspace{0.5cm}
\subfigure[]{
\resizebox*{!}{2.2cm}{\includegraphics{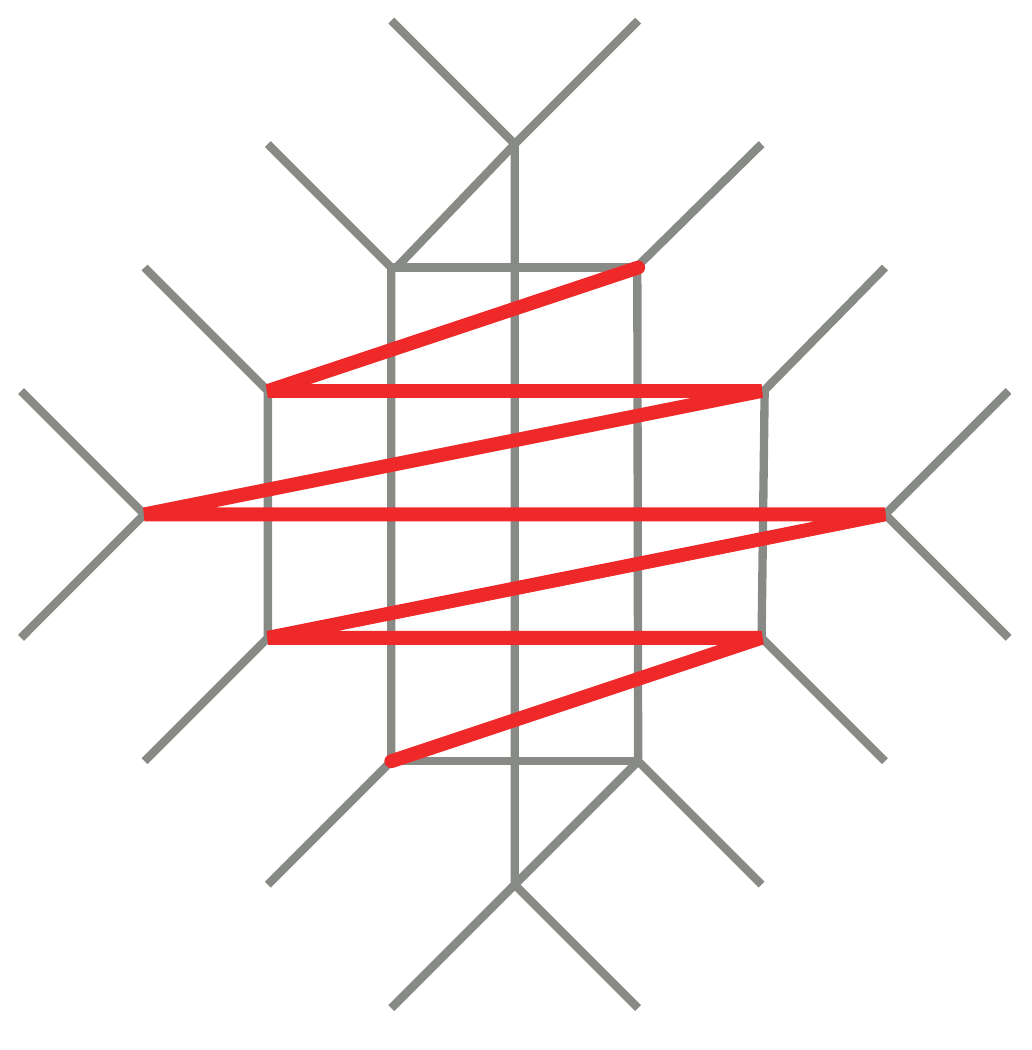}}}%
\hspace{0.5cm}
\subfigure[]{
\resizebox*{!}{2.2cm}{\includegraphics{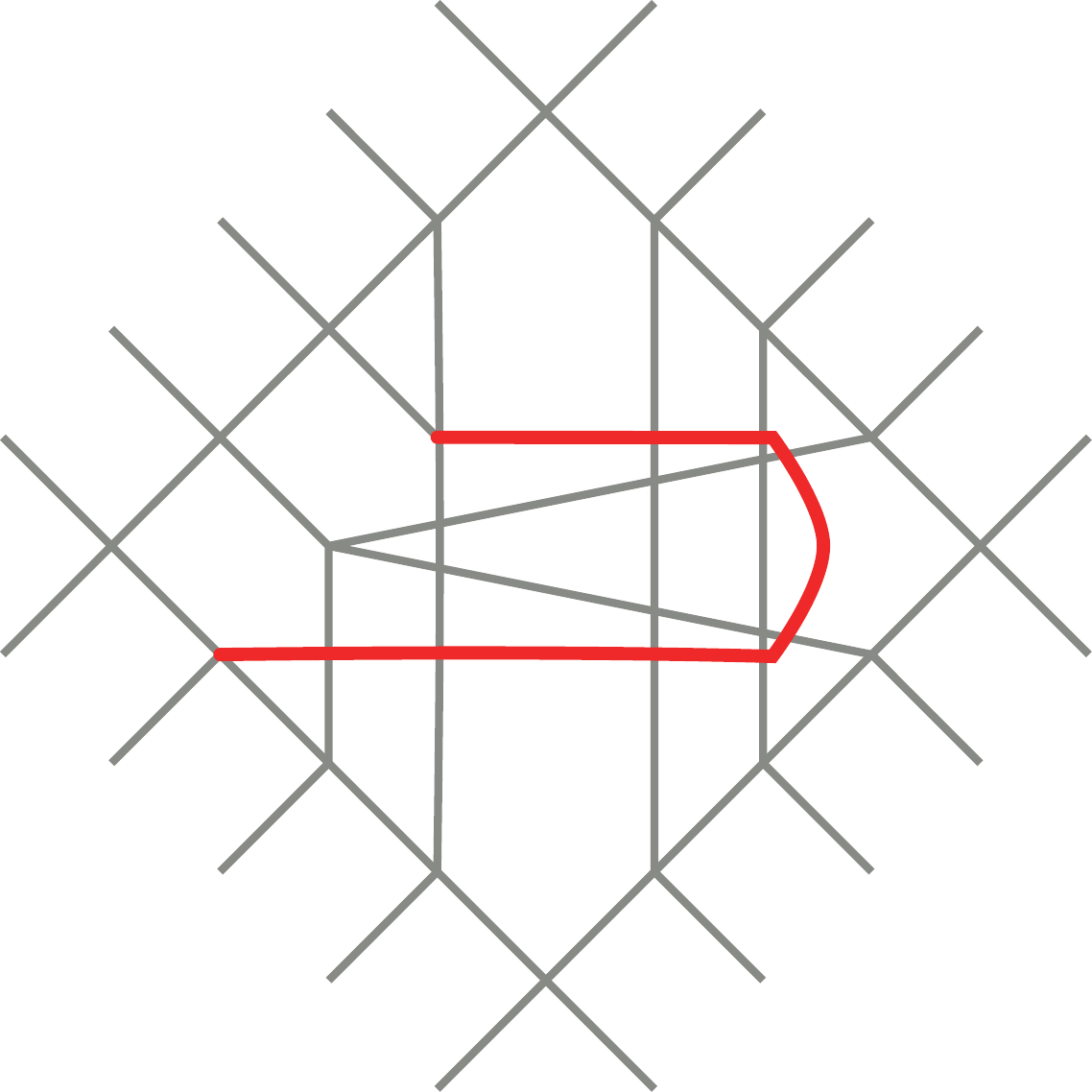}}}%
\end{center}
\caption{a) Spider b) Diamond c) Star. Bold lines indicate arcs not covered by lace paths.}
\label{fig:next}
\end{minipage}
\end{center}
\end{figure*}

The lace paths used in our evaluation do not cover some well known motifs such as the ones shown in Figure \ref{fig:next}.  The `spider' in Figure \ref{fig:next}a, could be generated by augmenting the set of step vectors considered (assuming the performance of the algorithm can be improved). The `diamond' and `star' shown in Figure \ref{fig:next}b and \ref{fig:next}c have a large number of crossings that are not on a lattice point and a different approach should be considered.

The catalogues, Viele G\"{u}te Gr\"{u}nde \cite{viele} in particular, contain grounds with over 1000 lattice points in the period parallelogram, well beyond the performance capability of our algorithm.  These large grounds typically possess $*2222$ or $*442$ orbifold symmetry \cite{conway}.

\begin{table}
\caption{Number of lace grounds reported by source.}
\label{table:catalogues}
\begin{center}
{\begin{tabular}{ll}
  \hline
  Source                                       & Number of Lace Grounds \\ \hline
  The Book of Bobbin Lace Stitches\cite{cook}  & 262 \\
  Gr\"{u}nde mit System\cite{system}           & 449 \\
  Viele G\"{u}te Gr\"{u}nde\cite{viele}        & 344 \\
  Algorithm                                    & $\sim100,000$ (no actions assumed) \\
  \hline
\end{tabular}}
\end{center}
\end{table}

Despite the limitations of a small step vector set and a restricted number of vertices, comparing the number of ground embeddings discovered algorithmically to the number of lace grounds recorded in the catalogues (see Table \ref{table:catalogues}) reveals that the algorithm has identified a large number of new patterns.  The catalogue counts in Table \ref{table:catalogues} include many examples where the same pair working diagram is used with several different cross-twist action combinations, each one of which is counted as a separate ground.  The algorithmic results in Table \ref{table:resulttable} strictly enumerate the number of ground embeddings with no assumption about a particular $\zeta$ function for the vertices.

It could be argued that the majority of new, algorithmically identified patterns lack symmetry and other elements necessary for aesthetic appeal.  To assess how interesting these new patterns may be to the community, we solicited feedback from lacemakers.  The ground patterns were announced on the user forum of the International Organization of Lace \cite{ningme} and made available through our web site \cite{uvicme}.  Lacemakers from the Czech Republic and Belgium have made small test designs using our results \cite{ningme} and some of the patterns are being included in a lace design software tool \cite{tool}.  We have received feedback that the asymmetric textures feel very organic and provide an interesting direction for modern designs.

 \begin{figure*}
\begin{center}
\begin{minipage}{\textwidth}
\begin{center}
\subfigure[]{
 \def\svgscale{0.42}
    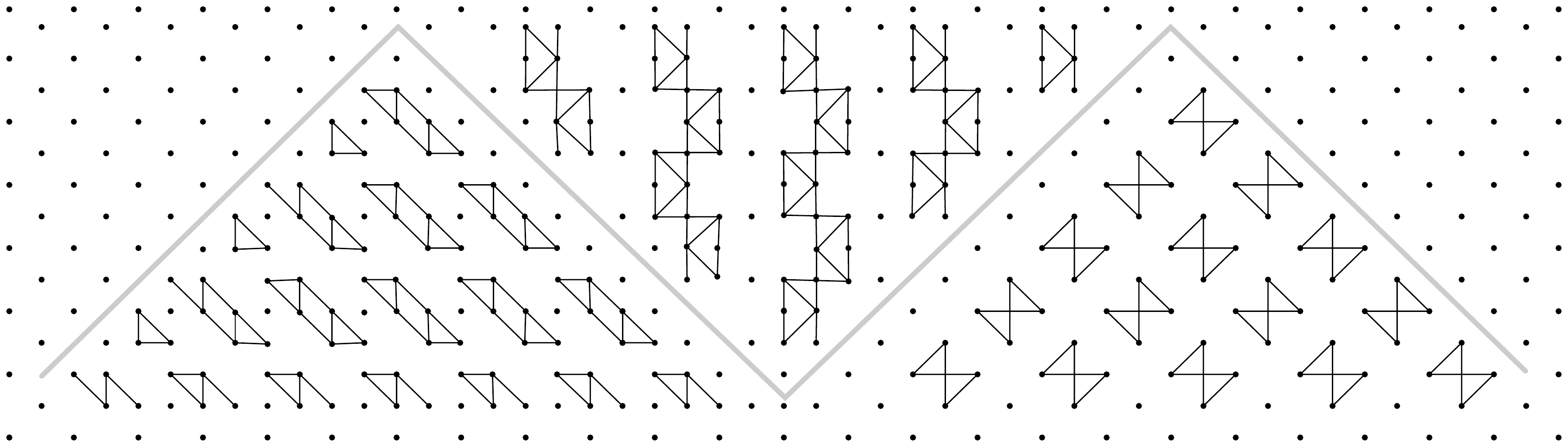 
}
\subfigure[]{
 \def\svgscale{0.9}
    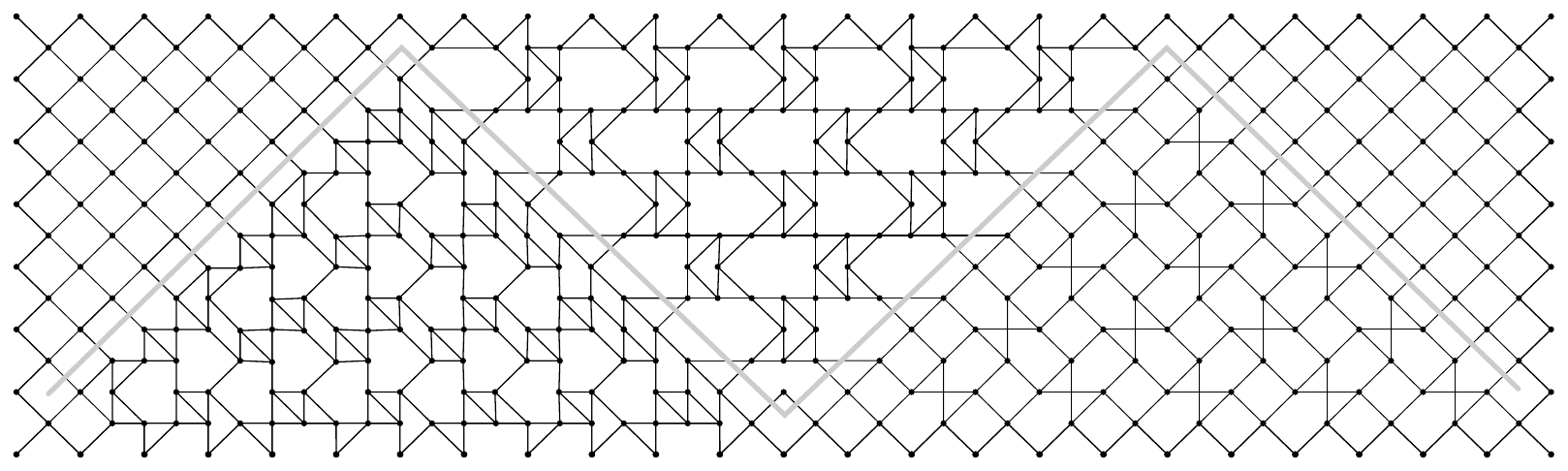 
}
\subfigure[]{
    \def\svgscale{0.48}
    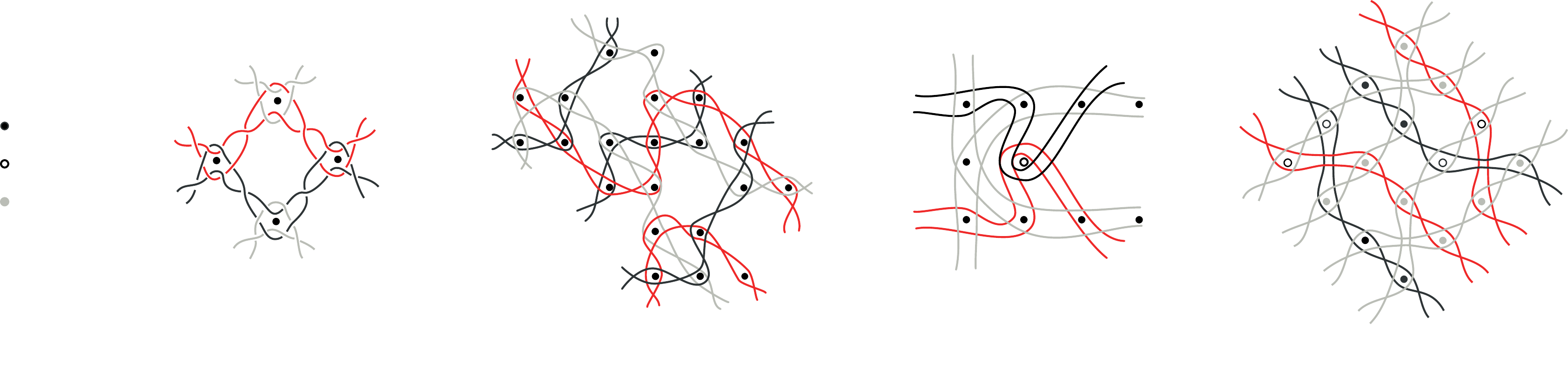 
}
\subfigure[]{
\resizebox*{\textwidth}{!}{\includegraphics{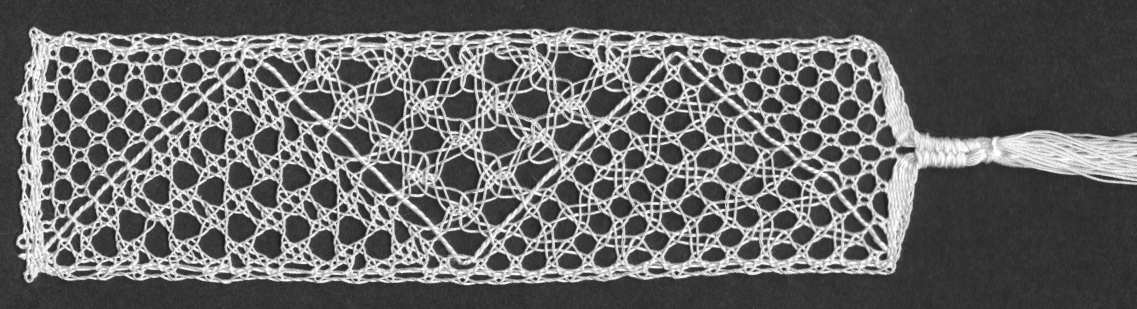}}} %
\end{center}
\caption{Note: Diagrams have been rotated $90^{\circ}$ counter clockwise. a) Pattern for bookmark using new grounds.
b) Pair working diagrams.
c) Thread working diagrams. A: $\zeta(v_1) = CTTpCT$, B: $\zeta(v_1) = CTpCT$, C: $\zeta(v_1) = CTpC$, $\zeta(v_2) = CTCpCTCT$, D: $\zeta(v_1) = CTpCT_{Left}$, $\zeta(v_2) = CTpCT_{Right}$, $\zeta(v_3) = CTpCT$.
d) Completed bookmark.  Actual size: 4 cm by 15 cm.  Made with 19 pairs of DMC Cordonnet Special No. 40. and DMC Coton Perl\'{e} No 5. }
\label{fig:bookmark2}
\end{minipage}
\end{center}
\end{figure*}

\section{Future Work}
\label{sec:future}

From a combinatorial point of view, the next step is to look at the size limitation imposed by the exponential growth of potential configurations and find a way to increase the vertex count of generated solutions.   Generating solutions with a certain symmetry type may significantly reduce the scope and therefore increase performance.
From the view point of an artist, there are already a large number of grounds to consider.  The embeddings need to be organized based on structural similarity and symmetry.  Also, as one can observe in the worked samples shown in Figures \ref{fig:worked} and \ref{fig:bookmark2}, it is not easy to predict how a particular $\zeta(v)$ mapping will affect the final appearance.  A tool that can help visualize these combinations would be invaluable.

\section*{Acknowledgments}
The authors would like to the thank the anonymous reviewers for their abundant, insightful comments and helpful suggestions.    Many thanks also to Jo Pol, Lorelei Halley and members of the International Organization of Lace community who have experimented with our ideas and provided feedback.

\bibliographystyle{plainnat}
\bibliography{mybib}

\end{document}